\documentclass[11pt]{article}

\usepackage{amsmath}
\usepackage{epsfig, amssymb, amsfonts}
\usepackage{amsthm}
\usepackage{amstext}
\usepackage{amsopn}
\usepackage{mathrsfs}
\usepackage{subfigure}
\usepackage{graphicx}
\usepackage[left=2.3cm,top=2cm,right=2.3cm,bottom=2cm, nohead,foot=1cm]{geometry}
\usepackage[makeroom]{cancel}
\usepackage{color}
\usepackage{enumerate}
\usepackage{comment}

\numberwithin{equation}{section}

\DeclareFontFamily{U}{mathx}{\hyphenchar\font45}
\DeclareFontShape{U}{mathx}{m}{n}{
      <5> <6> <7> <8> <9> <10>
      <10.95> <12> <14.4> <17.28> <20.74> <24.88>
      mathx10
      }{}
\DeclareSymbolFont{mathx}{U}{mathx}{m}{n}
\DeclareMathSymbol{\bigtimes}{1}{mathx}{"91}

\newtheorem{theorem}{Theorem}[section]

\newtheorem{corollary}[theorem]{Corollary}
\newtheorem{proposition}[theorem]{Proposition}

\theoremstyle{definition}
\newtheorem{definition}[theorem]{Definition}
\newtheorem{remark}[theorem]{Remark}

\newcommand{\R}{{\mathbb R}}

\newcommand{\p}{{p}}

\newcommand{\Z}{{\mathbb Z}}

\newcommand{\sbf}{{s}}
\newcommand{\bbf}{{b}}

\newcommand{\ds}{{\displaystyle}}

\date{\vspace{-9ex}}

\title{Continuum directed random polymers on disordered  hierarchical diamond lattices     }

\date{  }

  \author{ \textbf{Jeremy Thane Clark}\footnote{ {\tt
jeremy@olemiss.edu}} \vspace{.1cm}  \\  University of Mississippi, Department of Mathematics   }

\providecommand{\keywords}[1]{\textbf{\textit{Keywords:}} #1}

\begin{document}
\maketitle

\begin{abstract}
I discuss  models for a continuum  directed random polymer in a disordered environment in which the polymer lives on a fractal called the  \textit{diamond hierarchical lattice},  a  self-similar metric space forming a network of interweaving pathways.  This fractal depends on a branching parameter $b\in \mathbb{N}$ and a segmenting number $s\in \mathbb{N}$. For $s>b$ my focus is on random measures on the set of directed paths that can be formulated as a subcritical Gaussian multiplicative chaos.  This path measure is analogous to the continuum directed random polymer introduced  by Alberts,  Khanin,  Quastel [Journal of Statistical Physics \textbf{154}, 305-326 (2014)].

\end{abstract}


\section{Introduction}

Alberts, Khanin, and Quastel~\cite{alberts,alberts2} introduced a  \textit{continuum directed random polymer} (CDRP) model for a one-dimensional Wiener motion (the polymer) over a time interval $[0,1]$ whose law is randomly transformed through a field of impurities spread throughout the medium of the polymer.  The polymer's disordered environment  is  generated by  a time-space Gaussian white noise $\{\mathbf{W}(x)\}_{x\in D}$ where $ D:=[0,1]\times \R$ (in other terms, $\mathbf{W}$ is a  $\delta$-correlated Gaussian field). For  an inverse temperature parameter $\beta>0$, the CDRP is a random probability measure $Q_\beta^{\mathbf{W}}$ on the set of trajectories  $\Gamma :=C([0,1])$ that is formally expressed as
\begin{align}\label{Alberts}
Q^{\mathbf{W}}(dp)\,=\,  \frac{1}{ M(\Gamma )}  M(dp) \hspace{1cm} \text{for}\hspace{1cm}  M(dp )\,=\, e^{\beta\mathbf{W}_p  -\frac{\beta^2}{2}\mathbb{E}[\mathbf{W}_p ^2   ] }\mu(dp)\,,
\end{align}
where $\mu$ refers to the standard Wiener measure on $\Gamma$ and $\{\mathbf{W}_p\}_{p\in \Gamma}$ is a Gaussian field formally defined by integrating the white noise over a Brownian trajectory:  $\mathbf{W}_p=\int_0^1\mathbf{W}\big(r, p(r)  \big)dr$.  The random measure $M\equiv M(\mathbf{W})$ is a function of the field such that $\mathbb{E}[ M ]=\mu$ and yet $M$ is a.s.\ singular with respect to $\mu$.

 The rigorous mathematical   meaning of the random measure  $M$            in~(\ref{Alberts}) requires special consideration since exponentials of the field $\mathbf{W}_p $ do not have an immediately clear meaning, and, indeed, if the measure $M$ is singular with respect to $\mu$ the expression $\exp\{\beta\mathbf{W}_p  - (\beta^2 /2)\mathbb{E}[\mathbf{W}_p ^2   ] \}$ cannot define a Radon-Nikodym derivative $d M /d\mu $ anyway.  The construction approach of $M$ in~\cite{alberts2} involves an analysis of the finite-dimensional distributions through Wiener chaos expansions. Another point-of-view is that the  random measure $M$ has the form of a \textit{Gaussian multiplicative chaos} (GMC) measure over the Gaussian field $\{\mathbf{W}_p\}_{p\in \Gamma}$.  GMC theory began with an article by Kahane~\cite{Kahane} and much of the  progress on this topic has been motivated by the demands of quantum gravity theory~\cite{Barral,Duplantier,Duplantier2, Rhodes} although GMC theory arises in many other fields, including random matrix theory~\cite{Webb} and  number theory~\cite{Saksman}.  A GMC is classified as \textit{subcritical} or \textit{critical}, respectively, depending on whether  the expectation measure, $\mathbb{E}[M]$, is $\sigma$-finite or not. Because of its relevance to quantum gravity, the relatively unwieldy case of critical GMC has attracted the most attention, with recent results in~\cite{Beretycki,Junnila,Lacoin2}.   The random measure  $M$ in~(\ref{Alberts}) is subcritical since $\mathbb{E}[M  ]=\mu$ is a probability measure.  Shamov~\cite{Shamov}  has formulated subcritical GMC measure theory in a  particularly complete and accessible way.


In this article, I will study a GMC measure analogous to~(\ref{Alberts}) for a CDRP living on a fractal $D^{b,s}$ referred to as the \textit{diamond hierarchical lattice}. Given a branching parameter  $b\in \{2,3,\ldots\}$   and a segmenting parameter   $s\in \{2,3,\ldots\}$, the diamond hierarchical lattice is a compact metric space that embeds $bs$ shrunken copies of itself, which are arranged through $b$ branches that  each have $s$ copies running in series; see the construction outline  below in (A) and (B) of Section~\ref{SecOverview}.   Diamond hierarchical lattices provide a useful setting for formulating toy statistical mechanical models; for example, \cite{GLT,Goldstein,Griffiths, Hambly,HamblyII,lacoin3,Ruiz}.  Lacoin and Moreno~\cite{Lacoin} studied (discrete) directed polymers on disordered diamond hierarchical lattices, classifying the disorder behavior based on the cases $b<s$, $b=s$, and $b>s$, which are   combinatorially analogous, respectively, to the $d=1$, $d=2$, and $d>3$ cases of directed polymers on the $(1+d)$-rectangular lattice.    In~\cite{US},  we considered a functional limit theorem for the partition function of the $b<s$ diamond lattice polymer in a  scaling limit in which the temperature grows as a power law of the length of the polymer. This limit result is analogous to the intermediate disorder regime in~\cite{alberts} for   directed polymers on the $(1+1)$-rectangular lattice.  My focus here will be on developing the theory for a CDRP corresponding to the limiting partition function obtained in~\cite{US}. 

A similar  CDRP model on the $b=s$ diamond lattice, if it exists, will require a different approach for its construction; see~\cite{clark} for computations relevant to the continuum limit of discrete polymers.  It is interesting to compare this question with results~\cite{Carav1,Carav2}  by Caravenna, Sun, and Zygouras on scaling limits of the partition function for $(1+2)$-rectangular lattice polymers.

\subsection{Overview of the continuum directed random polymer on the diamond lattice }\label{SecOverview}

In this section I will sketch the construction of the continuum directed random polymer on the diamond hierarchical lattice and explore some of its properties.  I discuss the  diamond lattice fractal and its relevant substructures in more detail  in Section~\ref{SecGraphDef}.  Proofs of propositions are placed in Section~\ref{SecProofs}.

\vspace{.3cm}

\noindent \textbf{(A). The sequence of diamond hierarchical graphs}\vspace{.3cm}\\
For a branching number $b\in \{2,3,4,\ldots\}$ and a segmenting number $s\in \{2,3,4,\ldots\}$, the first diamond graph $D_{1}^{b,s}$ is defined by $b$ parallel branches connecting two root nodes $A$ and $B$ wherein each branch is formed by $s$  bonds running in series.  The  graphs  $D_{n+1}^{b,s}$, $n\geq 1$ are then constructed inductively by replacing each bond on $D_{1}^{b,s}$ by a nested copy of $D_{n}^{b,s}$; see the illustration below of the  $(b,s)=(2,3)$ case. 
\begin{center}
\includegraphics[scale=.6]{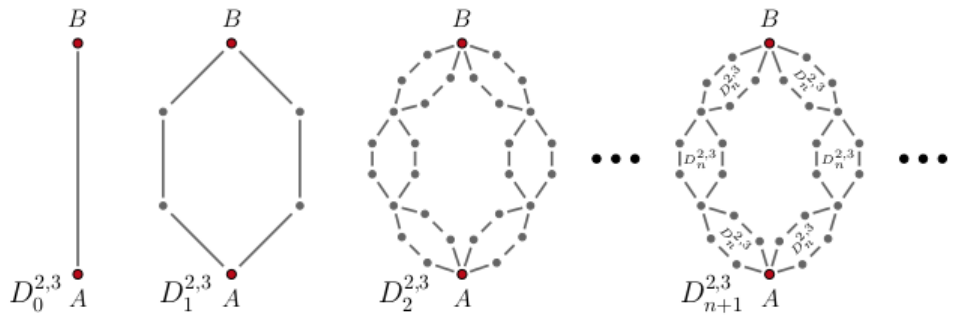}\\
\small  $D_{n+1}^{2,3}$ is defined through $6=2\cdot 3$ copies of $D_{n}^{2,3}$ that are connected in the formation of $D_1^{2,3}$.  
\end{center}

The set of bonds (edges) on the graph $D_{n}^{b,s}$ is denoted by $E_{n}^{b,s}$. A \textit{directed path} is a one-to-one function $\mathbf{p}: \{1,\ldots, s^n\}\rightarrow E_n^{b,s}$ such that  the bonds $\mathbf{p}(j)$, $\mathbf{p}(j+1)$ are adjacent for $1\leq j\leq s^n-1$   and the bonds $\textbf{p}(1)$ and $\textbf{p}(s^n)$ connect to $A$ and $B$, respectively. 
  I will use the following notations:
\begin{align*}
\vspace{.3cm}&V^{b,s}_n  \hspace{1.8cm}  \text{Set of vertex points on $D^{b,s}_n$}& \\
&E_{n}^{b,s}  \hspace{1.8cm}  \text{Set of bonds on the graph $D_{n}^{b,s}$}&     \\
&\Gamma_n^{b,s}  \hspace{1.85cm}  \text{Set of directed paths on $D_{n}^{b,s}$}&\\  
&[\textbf{p}]_n  \hspace{1.81cm}  \text{The path in $\Gamma_n^{b,s}$ determined by $\textbf{p}\in \Gamma^{b,s}_N $ for $N>n$}&
\end{align*}
\begin{remark}\label{RemarkHierarchy} The hierarchical structure of the sequence of diamond graphs implies that  $V^{b,s}_n$ is canonically embedded in  $V^{b,s}_N$  for $N>n$.   The vertices  in $V^{b,s}_n\backslash V^{b,s}_{n-1}$ are referred to as the $n^{th}$ \textit{generation} vertices. In the same vein, $E^{b,s}_n$ and $\Gamma_n^{b,s}$ define equivalence relations on $E^{b,s}_N$ and $\Gamma_N^{b,s}$, respectively.   For instance, $p,q\in \Gamma^{b,s}_N$ are equivalent up to generation $n$ if $[p]_n=[q]_n$. 

\end{remark}

\noindent \textbf{(B). The diamond hierarchical lattice}\vspace{.2cm}\\
Intuitively, the diamond hierarchical lattice, $D^{b,s}$, is a fractal  that  emerges as the ``limit" of the diamond graphs, $D_{n}^{b,s}$, as $n\rightarrow \infty$.  My convention is to view $D^{b,s}$ as a metric space embedding a family of interweaving copies of the interval $[0,1]$ for which the endpoints $0$ and $1$ are identified with the root vertices $A$ and $B$, respectively. In this framework directed paths are isometric maps $p:[0,1]\rightarrow D^{b,s}$ with $p(0)=A$ and $p(1)=B$.  In    Section~\ref{SecGraphDef} and Appendix~\ref{AppendCG}, I will give precise definitions for  $D^{b,s}$ and for the sets  $V^{b,s}$, $E^{b,s}$, $\Gamma^{b,s}$ and the measures  $\mu$, $\nu$  described in the notation list below. 
\begin{align*}
&V^{b,s}  \hspace{1.75cm}  \text{Set of vertex points on $D^{b,s}$}& \\
&E^{b,s}  \hspace{1.73cm}  \text{Complement of $V^{b,s}$ in   $D^{b,s} $  }& \\
&\Gamma^{b,s}  \hspace{1.8cm}  \text{Set of directed paths on $D^{b,s}$}& \\
&D^{b,s}_{i,j}\hspace{1.7cm}\text{First generation embedded copies of $D^{b,s}$ on the $j^{th}$ segment of the $i^{th}$ branch}& \\
&\nu     \hspace{2.2cm} \text{Uniform probability measure on $D^{b,s}$}\\
&\mu   \hspace{2.2cm} \text{Uniform probability measure on $\Gamma^{b,s}$}&
\\  
&[p]_n  \hspace{1.82cm}  \text{The path in $ \Gamma^{b,s}_n $ determined by $p\in \Gamma^{b,s}$}&
\end{align*}
\begin{remark}
$V^{b,s}$ is a countable, dense subset of $D^{b,s}$. 

\end{remark}

\begin{remark}\label{RemarkEmbed2} In analogy to Remark~\ref{RemarkHierarchy}, $V^{b,s}$ is canonically identifiable with $\cup_{n=1}^\infty V^{b,s}_n$.  Also, $E^{b,s}_n$ and   $\Gamma^{b,s}_n$ define equivalence relations on  $E^{b,s}$ and   $\Gamma^{b,s}$ for each $n\in \mathbb{N}$.

\end{remark}

\begin{remark} The measure $\nu$ is defined such that $\nu(V^{b,s})=0$ and, under the interpretation of Remark~\ref{RemarkEmbed2}, $\nu(\mathbf{e})=1/|E^{b,s}_n|$ for each $n\in \mathbb{N}$ and $\mathbf{e}\in E^{b,s}_n$.  Similarly $\mu$ is defined so that  $\mu(\mathbf{p})=1/|\Gamma^{b,s}_n|  $ for any $\mathbf{p}\in \Gamma^{b,s}_n$.
\end{remark}

\begin{remark}\label{RemarkConcat}
Let $\big(\Gamma^{b,s}_{i,j},\mu^{(i,j)}\big)$ be copies of $(\Gamma^{b,s},\mu)$ corresponding to the embedded subcopies, $D^{b,s}_{i,j}$, of  $D^{b,s}$.  The path space $(\Gamma^{b,s},\mu)$ can be decomposed as 
\begin{align}
\Gamma^{b,s}\,=\, \bigcup_{i=1}^{b}\bigtimes_{j=1}^s  \Gamma^{b,s}_{i,j} \hspace{1cm}\text{and}\hspace{1cm}    \mu\,=\,\,\frac{1}{b}\sum_{i=1}^b \prod_{j=1}^s \mu^{(i,j)}\,
\end{align}
by way of $s$-fold concatenation of the paths.   
\end{remark}

 I will assume that $b<s$ throughout the remainder of the text except where specified otherwise. 
\begin{proposition}\label{PropPathInter}
Fix some $p\in \Gamma^{b,s}$ and let $q\in \Gamma^{b,s}$ be chosen uniformly at random, i.e.,  according to the  measure $\mu(dq)  $.  Define the set of intersection times  $\mathcal{I}_{p,q}=\{r\in [0,1]\,|\,p(r)=q(r)   \}$, and  define $N_{p,q}^{(n)}:=\sum_{k=1}^{s^n} 1_{ [p]_n(k) = [q]_n(k)} $, in other terms, as the number of bonds shared by $[p]_n,[q]_n\in \Gamma^{b,s}_n$.  Finally, let $\frak{h}\in (0,1)$ be defined as   $\frak{h}:=\frac{\log s- \log b    }{ \log s  }$.

\begin{enumerate}[(i).]

\item  As $n\rightarrow \infty $ the sequence of random variables $T^{(n)}_{p,q}=(\frac{1}{s^n})^{\frak{h}}N_{p,q}^{(n)}$ converges a.s.\ to a limit $T_{p,q}$.  The moment generating function $\varphi_n(t)=\mathbb{E}\big[\exp\big\{ tT_{p,q}^{(n)}  \big\}  \big]$ converges pointwise to the moment generating function of $T_{p,q}$, and,  in particular, the second moment of $T^{(n)}_{p,q}$ converges to the second moment of $T_{p,q}$.

\item  $\mathcal{I}_{p,q}$ is a finite set with probability $1-\frak{p}_{b,s}$ for $\frak{p}_{b,s}\in (0,1)$ satisfying $\frak{p}_{b,s}=\frac{1}{b}\big[1-(1- \frak{p}_{b,s})^s  \big]$.  In this case, the intersections occur only at vertex points.

\item In the event that $\mathcal{I}_{p,q}$ is  infinite, the Hausdorff dimension  of $\mathcal{I}_{p,q}$ is a.s.\ $\frak{h}$.

\end{enumerate}

\end{proposition}

\begin{definition}
The  \textit{intersection time} $T_{p,q}$ of two paths $p,q\in \Gamma^{b,s}$  is defined as $$\displaystyle T_{p,q}\,:=\, \lim_{n\rightarrow \infty}\Big(\frac{1}{s^n}\Big)^{\frak{h}}N_{p,q}^{(n)} \,.   $$

\end{definition}

\begin{remark} The intersection time satisfies the formal identity
\begin{align}\label{LTime}
\int_0^1 \int_0^1\delta_D\big(p(r),q(t)\big)drdt   \,=\,T_{p,q} \,,
\end{align}
where $\delta_D$ is the  $\delta$-distribution on $D^{b,s}$ satisfying $f(x)=\int_{D^{b,s}  }\delta_D(x,y)f(y)\nu(y)$ for a test function $f:D^{b,s}\rightarrow \R$. The above identity can be understood in terms of the discrete graphs $D^{b,s}_n$ for which any two directed paths  $\mathbf{p},\mathbf{q}:\{1,\ldots, s^n\}\rightarrow E^{b,s}_n $ satisfy
 \begin{align}     
\frac{1}{s^{n}}\sum_{1\leq j\leq s^n}\frac{1}{s^{n}}\sum_{1\leq k\leq s^n} \frac{ 1_{ \mathbf{p}_j =\mathbf{q}_k}      }{  (bs)^{-n}   }  \,=\,& \Big(\frac{1}{s^n}\Big)^{\frak{h}} \sum_{1\leq j\leq s^n}  1_{ \mathbf{p}_j =\mathbf{q}_j}  \, =\, \Big(\frac{1}{s^n}\Big)^{\frak{h}} N^{(n)}_{p,q}  \,.
\end{align}
\end{remark}

\vspace{.5cm}

\noindent \textbf{(C).  A Gaussian field on directed paths}\vspace{.2cm}\\
 Let $\mathbf{W}$ denote a Gaussian white noise on $(D^{b,s},\nu)$  defined within some probability space $(\Omega,\mathcal{F},\mathbb{P})$, i.e., a linear map from $\mathcal{H}:=L^2(D^{b,s},\nu)$ into a Gaussian subspace of  $L^2(\Omega,\mathcal{F},\mathbb{P})$ such that
\begin{align*}
\psi\in  \mathcal{H}  \hspace{.5cm}\longmapsto\hspace{.5cm} \mathbf{W}(\psi)\,\sim\, \mathcal{N}\big(0,  \|\psi \|_{ \mathcal{H}  }^2\big)\,.
\end{align*}
I also use the alternative notations
$$ \mathbf{W}(\psi)\,\equiv\, \langle \mathbf{W} ,\, \psi\rangle  \,\equiv \,\int_{ D^{b,s}}\mathbf{W}(x) \psi(x)\nu(dx)\,,   $$
where the field  $\mathbf{W}(x)$,   $x\in D^{b,s}$ formally satisfies the $\delta$-correlation $\mathbb{E}[\mathbf{W}(x)\mathbf{W}(y)   \big]=\delta_D(x,y)$.

Next I discuss a field $\mathbf{W}_p$, $\p\in \Gamma^{b,s}$ formally defined by integrating the white noise along paths:
$$  \mathbf{W}_p \,=\,\int_0^1 \mathbf{W}\big(p(r)   \big) dr  \,.$$
The kernel $K_\Gamma(p,q)$ is equal to the intersection time, $T_{p,q} $, since~(\ref{LTime}) yields that
\begin{align*}
K_\Gamma(p,q)\,=\,\mathbb{E}\big[ \mathbf{W}_p  \mathbf{W}_q   \big] \,=\,\int_0^1 \int_0^1\delta_D\big(p(r),q(t)\big)drdt   \,=\,T_{p,q}    \,.
\end{align*}

Define the linear map $Y:\mathcal{H}\longrightarrow L^2(\Gamma^{b,s},\mu)  $  as
$$     (Y\psi)(p)\,=\,\int_0^1 \psi\big(p(r)\big)dr\,.  $$
To see that $Y$ is a bounded operator, notice that by Jensen's inequality 
$$\int_{ \Gamma^{b,s}  }\big|(Y\psi)(p)\big|^2\mu(dp)   \,\leq \,   \int_{\Gamma^{b,s}}\int_0^1 \big|\psi\big(p(r)\big)\big|^2dr\mu(dp) \,=\,\int_{D^{b,s}}\big|\psi(x)\big|^2\nu(dx)\,=\, \|\psi\|_\mathcal{H}^2\,,  $$
where the first equality  holds because the weight  assigned by $\nu$ to a set $R\subset D^{b,s}$ is equal to the expected amount of time that a  path $p\in \Gamma^{b,s}$ chosen uniformly at random spends in $R$; see  Proposition~\ref{PropUniform} and Remark~\ref{RemarkMuToNu}.  In particular, $Y$ is continuous when the topology of the codomain is identified with $L^0(\Gamma^{b,s},\mu)  $, i.e., convergence in measure with respect to $\mu$. In the terminology of~\cite{Shamov}, a continuous function $Y$ from $\mathcal{H}$  to  $L^0(\Gamma^{b,s},\mu)$ is referred to as a \textit{generalized $\mathcal{H}$-valued function} over the measure space $(\Gamma^{b,s},\mu)$.   The pair $(\mathbf{W},Y)$ encodes  the Gaussian field $\mathbf{W}_p$ on $ \Gamma^{b,s}$
by defining
\begin{align}\label{Tizle}
\int_{\Gamma^{b,s}}  \mathbf{W}_p f(p)\mu(dp)\,:=\, \langle \mathbf{W}, Y^*f\rangle\hspace{.4cm}\text{for a test function} \hspace{.4cm} f\in L^2\big(\Gamma^{b,s},\mu\big) \,. 
\end{align}
The adjoint $Y^*:L^2\big(\Gamma^{b,s},\mu\big)\rightarrow \mathcal{H}$ can be expressed in the form $(Y^*f)(x)=\int_{\Gamma^{b,s}}f(p)\int_0^1\delta_{D}( p(r), x)$.  I will use the  notation $ (Y\psi)(p) \equiv\langle  Y_p, \psi  \rangle $ and summarize~(\ref{Tizle}) as  $  \mathbf{W}_p\,=\, \langle Y_p , \mathbf{W}\rangle $.

\vspace{.5cm}

\noindent \textbf{(D). Gaussian multiplicative chaos on paths}\vspace{.1cm}\\
In this section, I discuss a random measure $M_{\beta}$ on $\Gamma^{b,s}$ formally related to $\mu$ and the field $\mathbf{W}_p$ through
\begin{align}
M_{\beta}(dp)\,=\, e^{ \beta \mathbf{W}_p-\frac{\beta^2}{2}\mathbb{E}[ \mathbf{W}_p^2]}\mu(dp)  \,.
\end{align}
The above does not define $ \textup{exp}\big\{ \beta \mathbf{W}_p-\frac{\beta^2}{2}\mathbb{E}[ \mathbf{W}_p^2]\big\}$ as a Radon-Nikodym derivative since $\{\mathbf{W}_p\}_{p \in \Gamma^{b,s}}$ is a Gaussian field rather than just an indexed family of centered Gaussian random variables.  
I state the proposition below using  Shamov's formulation of subcritical GMC measures~\cite{Shamov}.  

\begin{proposition}\label{ThmExist}
There exists a unique random measure, $M_{\beta}(dp)$, on $(\Gamma^{b,s},\mu)$ satisfying the properties (I)-(III) below.

\begin{enumerate}[(I).]

\item $\mathbb{E}[ M_{\beta}]= \mu$

\item $M_{\beta}$ is adapted to the white noise $\mathbf{W}$.  Thus I can write  $M_{\beta}(dp)\equiv M_{\beta}(\mathbf{W}, dp)$.  

\item  For $\psi\in \mathcal{H}$ and  a.e.\ realization of the field $\mathbf{W}$,
$$ M_{\beta}(\mathbf{W}+\psi, dp)\,=\,e^{\beta (Y\psi)(p) }M_{\beta}(\mathbf{W}, dp)   \,. $$

\end{enumerate}

\end{proposition}

\begin{remark}
 $M_\beta$ is the subcritical GMC on $(\Gamma^{b,s},\mu)$ over the field $(\mathbf{W}, Y  )$ with expectation $\mu$.
\end{remark}

\begin{remark}
I prove Proposition~\ref{ThmExist}  by constructing a sequence of GMC measures $\big\{M_{\beta}^{(n)}\big\}_{n\in \mathbb{N}}$ that form a martingale and have well-defined Radon-Nikodym derivatives $dM_{\beta}^{(n)}/d\mu $.
\end{remark}

 By~\cite{Shamov} the existence and uniqueness of the subcritical GMC measure $M_\beta$ in Proposition~\ref{ThmExist} 
is equivalent to the operator $\beta Y$ defining a \textit{random shift} of the field $\mathbf{W}$.  In other terms, the law $\widetilde{\mathbb{P}}_{\beta} $ determined by $ \mathcal{L}_{\widetilde{\mathbb{P}}_{\beta} } [ \mathbf{W}]:=\mathcal{L}_{\mathbb{P}\times \mu }  [\mathbf{W}+\beta Y_p] $
is absolutely continuous with respect to the law $\mathbb{P}$.

\begin{theorem}\label{ThmRandomShift} $\beta Y$ defines a random shift of the field $\mathbf{W}$.

\end{theorem}

\begin{theorem}\label{ThmProp}
For $\beta>0$ let the  random measure $M_{\beta}(dp)$ be defined as in Proposition~\ref{ThmExist}.
\begin{enumerate}[(i).]

\item $M_{\beta}$ is a.s.\ mutually singular to $\mu$.

\item The product measure $M_{\beta}\times M_{\beta}$ is a.s.\ supported on pairs $(p,q)\in \Gamma^{b,s}\times \Gamma^{b,s}$ such that the intersection set $\mathcal{I}_{p,q}=\{r\in[0,1]\,|\, p(r)=q(r)\}$ is either finite or has Hausdorff dimension $\frak{h}=(\log s-\log b)/\log s$.

\item Let $(\Gamma^{b,s}_{i,j},M_{\beta}^{(i,j)}) $ be independent copies of $(\Gamma^{b,s},M_{\beta}) $ corresponding to the first-generation embedded copies, $D^{b,s}_{i,j}$, of $D^{b,s}$.  Then there is equality in distribution of random measures
\begin{align*}
M_{\sqrt{\frac{s}{b}}\beta }\,\stackrel{d}{=}\, \frac{1}{b}\sum_{i=1}^b  \prod_{j=1}^s  M_{\beta }^{(i,j)} \hspace{.5cm}\text{under the identification}\hspace{.5cm} \Gamma^{b,s}\,\equiv\, \bigcup_{i=1}^{b}\bigtimes_{j=1}^s  \Gamma^{b,s}_{i,j} \,. 
\end{align*}

\end{enumerate}

\end{theorem}

\begin{remark}
Part (ii) of the Theorem~\ref{ThmProp} implies that the random measure $M_{\beta}$ a.s.\ has no atoms. 

\end{remark}

The next theorem states two strong disorder properties in the  $\beta\gg 1$ regime. Analogous results were obtained in~\cite{Lacoin} for discrete polymers on diamond graphs. 

\begin{theorem}\label{ThmDisorder}
Let the  random measure $M_{\beta}(dp)$ be defined as in Proposition~\ref{ThmExist}, and define the random probability measure $Q_\beta(dp)=M_\beta(dp)/M_\beta(\Gamma^{b,s})$.  As $\beta\rightarrow \infty$,

\begin{enumerate}[(i).]

\item the random variable $M_{\beta}(\Gamma^{b,s}   )$ converges in probability to $0$, and

\item  the random variable $\displaystyle \max_{\mathbf{p}\in \Gamma^{b,s}_n} Q_\beta(\mathbf{p})$ converges in probability to $1$ for each fixed $n$.

\end{enumerate}

\end{theorem}

\begin{remark}In particular, part (ii) implies that when $\beta\gg 1$  most of the weight of the measure $M_\beta(dp)$ is concentrated on a single coarse-grained path $\mathbf{p}\in \Gamma^{b,s}_n$.

\end{remark}

\vspace{.3cm}

\noindent \textbf{(E). A Gaussian multiplicative chaos martingale}\vspace{.1cm}\\
I will expand on the structure and properties of the map $Y:\mathcal{H}\rightarrow L^{2}(\Gamma^{b,s},\mu)$.

\begin{proposition}\label{PropYProp}  The linear operator $Y$ is compact and has the following properties:

\begin{enumerate}[(i).]

\item  $Y$ can be decomposed as $Y=UD$ where $U:\mathcal{H}\rightarrow L^{2}(\Gamma^{b,s},\mu)$ is an isometry and $D: \mathcal{H}\rightarrow \mathcal{H}$ is a self-adjoint operator with eigenvalues $\lambda_n= s^{-\frac{n-1}{2}}$ for $n\in \mathbb{N}\cup\{\infty\}$.

\item The null space of $Y$, which I denote by $\mathcal{H}_\infty$ ($\lambda_\infty=0$), is infinite dimensional.  For $2\leq n<\infty$,  the eigenspace $\mathcal{H}_n$ corresponding to the eigenvalue $\lambda_n$ has dimension $(bs)^{n-1}(b-1)$.  The eigenspace corresponding to $\lambda_0=1$ has dimension $b$, which  I    decompose into the one-dimensional space of constant functions, $\mathcal{H}_0$, and an orthogonal complement, $\mathcal{H}_1$.

\item  For $ 1\leq n<\infty$ the space $\mathcal{H}_n$ has an orthogonal basis $f_{(\mathbf{e},\ell)}$    labeled by $ (\mathbf{e},\ell)\in E^{b,s}_{n-1}\times \{1,\ldots, b-1\}$, where the function $f_{(\mathbf{e},\ell)}$ is supported on $\mathbf{e}\subset D^{b,s}$.

\item  $YY^*:L^2\big(\Gamma^{b,s},\mu)\rightarrow L^2\big(\Gamma^{b,s},\mu)$ is a self-adjoint operator with eigenvalues $\widehat{\lambda}_n= s^{-n+1}$ for $n\in \mathbb{N}\cup\{\infty\}$ with eigenspaces having dimension  $b$ for $n=1$ and $(bs)^{n-1}(b-1)$ for $n\geq 2$.  Hence $YY^*$ has Hilbert-Schmidt norm $\|YY^*\|_{HS}= \big(b\frac{s-1}{s-b} \big)^{1/2}$ but is not traceclass.

\end{enumerate}

\end{proposition}

\begin{remark} $\oplus_{k=0}^\infty \mathcal{H}_k$ is the orthogonal complement to the space of all $\psi \in  \mathcal{H}= L^2(D^{b,s},\nu) $   such that $\int_{0}^{1}\psi\big(p(r)  \big)dr$ is zero for every path $p\in \Gamma^{b,s}$.

\end{remark}

\begin{definition}\label{DefYNs}
Define $Y^{(n)}:\mathcal{H}\rightarrow L^{2}\big(\Gamma^{b,s},\mu   \big)$ to act as $ Y^{(n)}\psi =\textup{E}\big[ Y\psi   \,\big|\, \textup{F}_n  \big]$, where $\textup{F}_n$ is the $\sigma$-algebra on $\Gamma^{b,s}$  generated by the map $p\mapsto [p]_n$.

\end{definition}

\begin{remark}\label{YRemark}$Y^{(n)}$ can also be written in the following forms:

\begin{itemize}

\item

 $ \big(Y^{(n)}\psi\big)(p)  = \big|\Gamma_n^{b,s}\big|  \int_{\tilde{p}\in [p]_n }   \big(Y\psi\big)(\tilde{p}) \mu(d\tilde{p})$, where  the right side can be understood as the average of $\big(Y\psi\big)(\tilde{p})$ over all $\tilde{p}\in \Gamma^{b,s}$ in the same generation-$n$ equivalence class of $p$.

\item   $
 \big(Y^{(n)}\psi\big)(p)\,=\, \big\langle Y^{(n)}_p, \psi \big\rangle $ for  $Y^{(n)}_p\in \mathcal{H}$ defined as   $Y^{(n)}_p\,:=\, \chi_{ T_p^{(n)}   }/\nu\big( T_p^{(n)}  \big)$,
where $ T_p^{(n)}=\cup_{k=1}^{s^n}[p]_n(k) $, i.e., the generation-$n$ coarse-grained trace of the path $p$ through the space $D^{b,s}$.

\end{itemize}

\end{remark}

\begin{proposition}\label{PropYs} The maps $Y^{(n)}:\mathcal{H}\rightarrow L^{2}\big(\Gamma^{b,s},\mu   \big)$ satisfy the properties below.

\begin{enumerate}[(i).]

\item  Let $\mathbf{P}_n:\mathcal{H}\rightarrow \mathcal{H}$ be the orthogonal projection onto $\oplus_{k=0}^n \mathcal{H}_k$ for $\mathcal{H}_k$  defined as in part (ii) of Proposition~\ref{PropYProp}.
   For any $\psi\in \mathcal{H}$,
$$Y^{(n)}\psi \,=\, Y\mathbf{P}_n \psi   \,. $$

\item  As $n\rightarrow \infty$, the map $Y^{(n)}$ converges in  operator norm to $Y$.

\item  As $n\rightarrow \infty$,  $ Y^{(n)} (Y^{(n)})^*$ converges in  Hilbert-Schmidt norm to $ Y Y^*$, which has integral kernel $K_\Gamma(p,q)=T_{p,q}  $, i.e., the intersection time of the paths.

\item   For any $k\in \mathbb{N}$ and $p \in \Gamma^{b,s}$, $ Y^{(k)}_p -Y^{(k-1)}_p\in \mathcal{H}_{k}$.  In particular, the following sequence of vectors in $ \mathcal{H}$  are orthogonal:
\begin{align}\label{Ortho}
Y^{(0)}_p\,,\,\,\, Y^{(1)}_p -Y^{(0)}_p\,,\,\, \, Y^{(2)}_p -Y^{(1)}_p\,,\,\,\ldots  
\end{align}

\end{enumerate}

\end{proposition}

\begin{proposition}\label{PropGMCApprox} Define  $\mathcal{F}_{n}$ to be the  $\sigma$-algebra on $\Omega$ generated by the field variables $\langle \mathbf{W}, \psi\rangle$ for $\psi\in \oplus_{k=0}^n \mathcal{H}_k $.  Let $M^{(n)}_\beta$  be the GMC measure over the finite-dimensional field $(\mathbf{W}, \beta Y^{(n)})$, i.e., with Radon-Nikodym derivative $$\frac{d M^{(n)}_\beta}{d\mu }=\textup{exp}\Big\{ \beta \langle \mathbf{W}, Y_p^{(n)} \rangle -\frac{\beta^2}{2}\| Y_p^{(n)} \|_{\mathcal{H}}^2   \Big\}\,.$$
The sequence of measures $\big\{M^{(n)}_\beta\big\}_{n\in\mathbb{N}}$ forms a martingale with respect to the filtration  $\mathcal{F}_{n}$ and  a.s.\ converges vaguely  to the GMC measure $M_\beta$.

\end{proposition}

\vspace{.3cm}

\noindent \textbf{(F). Chaos expansion construction of the GMC measure}\vspace{.1cm}\\
The Gaussian multiplicative chaos $M_\beta$ can also be constructed through chaos expansions analogous to those in~\cite{alberts2}.  The \textit{chaos decomposition} generated by the field $\mathbf{W}$   is 
\begin{align}\label{Dizzle}
L^{2}\big(\Omega,\mathcal{F}(\mathbf{W}),\mathbb{P}\big)\,=\,\bigoplus_{k=0}^{\infty}\mathcal{H}^{:k:}_{\mathbf{W}}\,,
\end{align}
where $\mathcal{H}^{:k:}_{\mathbf{W}}$ is the orthogonal complement of the set $\overline{\mathcal{P}}_{k-1}(\mathcal{H})$ within $\overline{\mathcal{P}}_{k}(\mathcal{H})$ for
$$\mathcal{P}_k(\mathcal{H})\, := \,\big\{ p\big(\langle \mathbf{W}, \psi_1 \rangle,\ldots, \langle \mathbf{W}, \psi_k \rangle     \big) \,\big|\, \text{$p:\R^k\rightarrow \R$ is a degree-$k$ polynomial and }  \psi_1,\ldots ,\psi_k \in \mathcal{H}   \big\} \,. $$
There is a canonical isometry between $\mathcal{H}^{:k:}_{\mathbf{W}}$ and the $k$-fold symmetric tensor,  $\mathcal{H}^{\odot k}$, of $\mathcal{H}$.  This isometry can be expressed as a map from symmetric  functions $f(x_1,\ldots,x_k)$ in $L^{2}\big((D^{b,s})^k,\frac{1}{k!}\nu^k\big)$ to elements of $\mathcal{H}^{:k:}_{\mathbf{W}}$ expressed as stochastic integrals:
\begin{align}
\frac{1}{k!}\int_{ (D^{b,s})^k }f(x_1,\ldots,x_k)\mathbf{W}(x_1)\cdots \mathbf{W}(x_k) \nu(dx_1)\cdots \nu(dx_k)\,.
\end{align}
Intuitively, the integral above is to be understood as over points $(x_1,\ldots, x_k)\in (D^{b,s})^k$ for which the components $x_j$ are distinct. See~\cite[Section 7.2]{Janson} for the general theory of stochastic integrals.  

Let $S$ be a finite subset of $E^{b,s}$ and define $\Gamma^{b,s}_{S}$ as the collection of paths $p\in \Gamma^{b,s}$ such that $S\subset\textup{Range}(p)$.   If $\Gamma^{b,s}_{S}$ is nonempty, define $\mu_{S}$ as the probability measure  
uniformly distributed over paths in $\Gamma^{b,s}_{S}$ (and thus supported on $\Gamma^{b,s}_{S}$). If $\Gamma^{b,s}_{S}=\emptyset$, i.e., there is no path containing all the points in $S$, then $\mu_{S}$ is defined as zero.

\begin{definition}\label{DefRho} For a Borel set $A\subset \Gamma^{b,s}$, define $\rho_{k}(x_1,\ldots,x_k; A)$ as a map from $(D^{b,s})^k$ to $[0,\infty)$  with 
$$\rho_k(x_1,\ldots, x_k; A)\,:=\, \begin{cases} b^{\gamma(\{ x_1,\ldots, x_k \})}\mu_{\{x_1,\ldots,x_k\}}(A)  & \text{ $ x_1,\ldots, x_k \in  E^{b,s}$ are distinct,}   \\  0 & \text{otherwise,}  \end{cases}   $$ where $\gamma (S)$ is defined for a finite subset of $E^{b,s}$ as  $$\ds
 \gamma (S)\,:=\,\sum_{k=0}^{\infty} \Big(|S|\,-\, \big| \big\{\mathbf{e} \in E_{k}^{b,s}\,\big|\, \mathbf{e}\cap S \neq \emptyset  \big\}\big| \Big)\,. $$ 
In the above formula for $\gamma (S)$, elements of $E_{k}^{b,s}$ are to be understood as subsets of $E^{b,s}$, and the $k=0$ term of the sum is always interpreted as $|S|-1$.

\end{definition}

\begin{remark}
The term $ \big| \big\{\mathbf{e} \in E_{k}^{b,s}\,\big|\, \mathbf{e}\cap S \neq \emptyset  \big\}\big|$ counts the number of distinct equivalence classes from $E_k^{b,s}$ corresponding to elements in $S$.

\end{remark}


\begin{theorem}\label{ThmChaosExp} For any Borel set $A\subset \Gamma^{b,s}$, the random variable $M_\beta(\mathbf{W}, A)$ is  equal to the chaos expansion
\begin{align*}
\mu(A)\,+\, \sum_{k=1}^{\infty} \frac{\beta^k}{k!}\int_{(D^{b,s})^k}\rho_k(x_1,\ldots, x_k; A)\mathbf{W}(x_1)\cdots \mathbf{W}(x_k)\nu(dx_1)\cdots \nu(dx_k)\,.
\end{align*}
 The sequence of symmetric functions $\{\rho_k(x_1,\ldots, x_k; A)\}_{k\in \mathbb{N}}$ satisfies
\begin{enumerate}[(i).]
\item $\int_{D^{b,s}}\rho_k\big(x_1,\ldots, x_k; A\big) \nu(dx_k)\,=\,\rho_{k-1}(x_1,\ldots, x_{k-1}; A) $ and
 
\item $\int_{(D^{b,s})^k}\rho_k\big(x_1,\ldots, x_k; A\big)\nu(dx_1)\cdots \nu(dx_k)\,=\,\mu(A)$.
\end{enumerate}

\end{theorem}

\vspace{.3cm}

\noindent \textbf{(G). Scaling limits from non Gaussian variables}\vspace{.1cm}\\
For each $n\in \mathbb{N}$ let $\big\{\omega_{a}^{(n)}\big\}_{a\in E_n^{b,s}}$ be a family of i.i.d.\ random variables indexed by the edge set of the diamond lattice $D_n^{b,s}$.  Assume that the variables have mean zero, variance one, and finite exponential moments: $\mathbb{E}\big[\textup{exp}\big(\beta \omega_a^{(n)}\big)   \big]<\infty$ for $\beta \in \R$.  Define a random measure $\mathbf{M}_{\beta}^{(n)}$  on $ \Gamma^{b,s}_n$ as follows:
\begin{align}
\mathbf{M}_{\beta}^{(n)}(A)  \,=\, \frac{1}{\big|\Gamma^{b,s}_n\big|}\sum_{\mathbf{p}\in A} \prod_{k=1}^{s^n}\frac{ \exp\left\{\beta\omega_{\mathbf{p}(k)}^{(n)}\right\}   }{ \mathbb{E}\Big[\exp\left\{\beta\omega_{\mathbf{p}(k)}^{(n)}\right\} \Big]   } \hspace{.5cm}\text{for}\hspace{.5cm}A\subset \Gamma^{b,s}_n\,.
\end{align}

The theorem below follows as a consequence of Theorem 4.6 of~\cite{US}.
\begin{theorem}\label{ThmUniversal} Fix $N\in \mathbb{N}$ and  let subsets of $\Gamma_N^{b,s}$ be identified with the canonically corresponding subsets of $\Gamma^{b,s}$ and of $\Gamma_n^{b,s}$ for $n>N$.  For $\beta_n:=\beta(b/s)^{n/2}$, the family of random variables $\mathbf{M}_{\beta_n}^{(n)}(A)$ labeled by $A\subset \Gamma_N^{b,s}$ has joint convergence in law as $n\rightarrow \infty$ given by
\begin{align}
\big\{\mathbf{M}_{\beta_n}^{(n)}(A)\big\}_{A\subset \Gamma_N^{b,s}}\hspace{.5cm}\stackrel{\mathcal{L}}{\Longrightarrow} \hspace{.5cm} \big\{ M_{\beta}(A)\big\}_{A\subset \Gamma_N^{b,s}}\,.
\end{align}

\end{theorem}

\section{Diamond hierarchical lattice  }\label{SecGraphDef}

In this section I will provide a path-based construction of the diamond hierarchical lattice as a compact metric space.  This discussion  is applicable to any choice of $b,s\in  \{2,3,\ldots\}$. The  proofs of propositions  are in the appendix.

\subsection{Construction of the diamond lattice as a metric space}

The hierarchical formulation of the diamond graph $D_n^{b,s}$ in terms of $b\cdot s$ embedded copies of $D_{n-1}^{b,s}$ carries with it a canonical one-to-one correspondence between  $(\{1,\ldots, b\}\times \{1,\ldots, s\})^n $
and the set of bonds, $E_n^{b,s}$.  I will construct the diamond hierarchical lattice, $D^{b,s}$, as an equivalence relation on the set of sequences $ \mathcal{D}^{\bbf,\sbf}\,:=\,\big( \{1,\ldots, b\}\times \{1,\ldots, s\}\big)^{\infty} $ determined  by a semi-metric $d_D:\mathcal{D}^{\bbf,\sbf}\times \mathcal{D}^{\bbf,\sbf}\rightarrow [0,1]$ defined below.

Let  $\widetilde{\pi}: \mathcal{D}^{\bbf,\sbf}\rightarrow [0,1]$ be the ``projective" map  sending a sequence $x= \{(b_k^x,s_k^x)\}_{k\in \mathbb{N}}$ to
$$\widetilde{\pi}(x)\,:=\, \sum_{k=1}^{\infty} \frac{s_{k}^{x}-1}{ s^{k}}  \,.$$
Of course, the right side above is the base $s$ generalized decimal expansion of the number  $\widetilde{\pi}(x)\in [0,1]$ having $k^{th}$ digit $s_k^x-1$. The root vertices of the continuum  lattice will be  identified with the sets $A:=\{x\in \mathcal{D}^{\bbf,\sbf}\,|\, \widetilde{\pi}(x)=0 \}$ and $B:=\{x\in \mathcal{D}^{\bbf,\sbf}\,|\, \widetilde{\pi}(x)=1 \}$.

For two points  $x=\{(b_k^x,s_k^x)\}_{k\in \mathbb{N}}$ and  $y=\{(b_k^y,s_k^y)\}_{k\in \mathbb{N}}$ in $\mathcal{D}^{\bbf,\sbf}$, I write $x\updownarrow y$ if $x$ or $y$ is contained in $A\cup B$  or  for some $n\in \mathbb{N}$
$$  (b_k^x,s_k^x)=(b_k^y,s_k^y)\,\, \text{ for }\,\, 1\leq k < n-1 \hspace{.5cm}\text{and}\hspace{.5cm} b_n^x=b_n^y \hspace{.5cm} \text{but}\hspace{.5cm}  s_n^x\neq s_n^y \,.    $$
In other terms the sequence of pairs defining $x$ and $y$ disagrees for the first time at  an  $s$-component value.  Intuitively, this means that there exists a directed path going through both  $x$ and $y$.  We then define the semi-metric $d_D$ on $\mathcal{D}^{\bbf,\sbf}$ as the traveling distance
$$d_D(x,y)\,:=\,\begin{cases} \quad  \quad  \big|\widetilde{\pi}(x)-\widetilde{\pi}(y)\big|   & \quad\text{if } x\updownarrow y,  \\ \,\, \displaystyle \inf_{z \in \mathcal{D}^{\bbf,\sbf},\, z\updownarrow x,\, z\updownarrow y  }\Big( d_D(x,z)+d_D(z,y) \Big)  & \quad  \text{otherwise.} \end{cases}     $$
 The semi-metric    is bounded by $1$ since by choosing  appropriate $z\in A$ or $z\in B$ in the infimum above,  I can conclude that $d_D(x,y)\leq \min\big(\widetilde{\pi}(x)+\widetilde{\pi}(y) ,2-\widetilde{\pi}(x)-\widetilde{\pi}(y)\big)$.
 
\begin{definition}\label{DefCDL}
 The \textit{diamond hierarchical lattice}  is defined as 
$$D^{b,s}\, := \, \mathcal{D}^{\bbf,\sbf} /\big(x,y\in \mathcal{D}^{\bbf,\sbf} \text{ with }  d_D(x,y )=0\big)\,.   $$
In future, I will treat the metric $d_D(x,y)$ and the map $\widetilde{\pi}$  as acting on $D^{b,s}$.  
\end{definition}

\begin{remark}\label{RemarkVertex}
 The  vertex set, $V_n^{b,s}$, on the diamond graph $D_n^{b,s}$ is canonically embedded on $D^{b,s}$; see  Appendix~\ref{AppendCG}.   The representation of elements in $D^{b,s}$ by sequences $\{(b_j,s_j)\}_{j\in \mathbb{N}}  $ is unique except for the countable collection of vertices $V^{b,s}:=\bigcup_n V^{b,s}_n $.  
\end{remark}

\begin{remark}\label{RemarkFractal} The self-similar structure of the fractal $D^{b,s}$ can be understood through a family of contractive shift maps $S_{i,j}:\mathcal{D}^{b,s}\rightarrow \mathcal{D}^{b,s}$ for $(i,j)\in \{1,\ldots, b\}\times \{1,\ldots,s\}$ that
 send $x=\{(b_k^x,s_k^x)\}_{k\in \mathbb{N}}$ to $S_{i,j}(x)=y=\{(b_k^y,s_k^y)\}_{k\in \mathbb{N}}$ with $(b_1^y,s_1^y)=(i,j)$ and $(b_k^y,s_k^y)=(b_{k-1}^x,s_{k-1}^x)$ for $k\geq 2$.  The $S_{i,j}$'s are well-defined as functions on  $D^{b,s}$, and map $D^{b,s}$ onto the shrunken subcopies $D^{b,s}_{i,j}$ with
$$\text{ }\hspace{2.3cm} d_D\big(S_{i,j}(x),S_{i,j}(y)\big)\,=\,\frac{1}{s} d_D(x,y) \,, \hspace{1cm} x,y\in D^{b,s} \,. $$

\end{remark}

\begin{proposition}\label{PropCompact}  $(D^{b,s},d_D )$ is a compact metric space with Hausdorff dimension $1+\frac{\log b}{\log s}$.  The vertex set $V^{b,s}$ is dense in $D^{b,s}$.

\end{proposition}

The import of the next proposition is that a probability measure $\nu$ can be placed on $D^{b,s}$ such that subsets identifiable with elements of $E_n^{b,s}$ are assigned measure $ (bs)^{-n} $.

\begin{proposition}\label{PropCylinder} 
Let $\mathcal{B}_D$ be the Borel  $\sigma$-algebra on $D^{b,s}$ generated by the metric $d_D$. There is a unique  measure $\nu$ on $(D^{b,s},\mathcal{B}_D )$ such that $\nu(V^{b,s})=0$ and     for $(b_j,s_j)\in \{1,\ldots, b\}\times \{1,\ldots, s\}$ the  cylinder sets
\begin{align*}
C_{ (b_1,s_1)\times \cdots \times (b_n,s_n)}  \,:=\, \bigg\{x\in  E^{b,s} \,\bigg|\,  x=\{ (b_j^x, s_j^x)\}_{j\in \mathbb{N}}  \,\, \text{with} \,\, b_j^x=b_j \,\,\text{and}\,\, s_j^x=s_j \text{ for } 1\leq j\leq n \bigg\}
\end{align*}
(identifiable with elements in $E_n^{b,s}$)  have measure $\nu\big(C_{ (b_1,s_1)\times \cdots \times (b_n,s_n)}  \big)=|E_n^{b,s}  |^{-1}= (bs)^{-n}$.
\end{proposition}

\subsection{Directed paths on the diamond hierarchical lattice}
I define a  \textit{directed path}   on $D^{b,s}$ to be a continuous function $p:[0,1]\rightarrow D^{b,s}$ such that $\widetilde{\pi}\big( p(r)\big)=r $ for all $r\in[0,1]$.  I will use the uniform metric on the set of directed paths:
$$    d_\Gamma\big(p_{1},p_{2}\big) \, = \, \max_{0\leq r\leq 1}d_D\big( p_{1}(r), p_{2}(r)    \big) \hspace{1cm}p_{1},p_{2}\in \Gamma^{b,s}   \,.     $$

\begin{remark}Note that $ d_\Gamma\big(p_{1},p_{2}\big)\, = \,  s^{-(n-1)}      $, where $n\in \mathbb{N}$ is the lowest generation such that there is a  vertex $v\in V_{n}^{b,s}$ in the range of $ p_{1}$ but not of $ p_{2}$. 
\end{remark}

\begin{remark}
 $\Gamma^{b,s}_n$ is canonically identified with an equivalence relation of $\Gamma^{b,s}$ in which $q\equiv_n p$ iff  $[p]_n=[q]_n$, or, equivalently, $d_\Gamma(p,q)\leq s^{-n}$.
\end{remark}

\begin{remark}\label{RemarkMetric}
The metric $d_D$ on $D^{b,s}$ can be  reformulated in terms of the space of directed paths, $\Gamma^{b,s}$, as
$$ d_D(x,y)\,=\,\inf_{ \substack{ p,q\in \Gamma^{b,s}, z\in D^{b,s} \\  z\in \textup{Range}(p)\cap  \textup{Range}(q) }  } \Big(|\widetilde{\pi}(x)- \widetilde{\pi}(z)|+|\widetilde{\pi}(z)- \widetilde{\pi}(y)|  \Big) \,. $$  
\end{remark}

The \textit{uniform measure} on  $\Gamma^{b,s}$  refers to the triple      $\big(\Gamma^{b,s},\mathcal{B}_\Gamma,  \mu\big)$ in the proposition below.

 \begin{proposition}\label{PropUniform} Let $\mathcal{B}_\Gamma $  be the Borel $\sigma$-algebra generated by the metric $d_\Gamma$.  There is a unique measure $\mu$ on $(\Gamma^{b,s},\mathcal{B}_\Gamma)$ satisfying $\mu(\mathbf{p})=|  \Gamma^{b,s}_n |^{-1}=b^{-\frac{s^{n}-1}{s-1}}   $ for all $n\in \mathbb{N}$ and  $\mathbf{p}\in \Gamma^{b,s}_n$.  Moreover, $\mu$ is related to $\nu$  through the following identity: for any $R\in \mathcal{B}_D$
\begin{align}\label{MuToNu} 
  \nu(R)\,=\,   \int_{\Gamma^{b,s}}\int_{[0,1]}1_{p(r)\in R}\,dr\mu(dp)   \,.  
  \end{align}

 \end{proposition}

\begin{remark}\label{RemarkMuToNu} The intuitive meaning of~(\ref{MuToNu}) is that $\nu(R)$ equals the expected amount of time that a random path $p\in \Gamma^{b,s}$ chosen according to the measure $\mu$ will spend in $R\in \mathcal{B}_D $. This implies, more generally, that for any $f\in L^{1}(D^{b,s},\nu)$
$$ \int_{D^{b,s}  }f( x )\nu(dx) \,=\, \int_{\Gamma^{b,s}}\int_{[0,1]}f\big( p(r)\big)dr\mu(dp)\,.  $$

\end{remark}

\section{Proofs}\label{SecProofs}

\subsection{Intersection time between random directed paths   }

The intersections between randomly chosen paths $p,q \in \Gamma^{b,s}$ can be encoded into  realizations of a discrete-time branching process that begins with a single node and for which each generation $n$ node has exactly $s$ children independently with probability $1/b$, or has no children at all. 

Given $p,q\in \Gamma^{b,s}$ recall that $N_{p,q}^{(n)}$ is  the number of bonds shared by the coarse-grained paths $[p]_n,[q]_n\in \Gamma^{b,s}_n$.  For $q\in  \Gamma^{b,s}$ chosen at random (i.e., according to $\mu$), let  $\textup{F}_n^{(q)}$ be the $\sigma$-algebra of subsets of $\Gamma^{b,s}$ generated by $[q  ]_n$.

\begin{proof}[Proof of Proposition~\ref{PropPathInter}]  We can write $I_{p,q}$ as
$$I_{p,q}\, =\,\bigcap_{n=1}^\infty I_{p,q}^{(n)}\hspace{1cm}\text{for}\hspace{1cm} I_{p,q}^{(n)}\,:=\,[0,1]\,-\,\bigcup_{\substack{1\leq k \leq s^n \\  [p]_n(k)\neq [q]_n(k)   }}\Big(\frac{k-1}{s^n}, \frac{k}{s^n}  \Big)  \,. $$
 Let $\frak{p}_{b,s}^{(n)}$ be the probability that the number, $N_{p,q}^{(n)}$,   of bonds shared by $[p]_n $ and $[q]_n$  is not zero.  Then the probability that $N_{p,q}^{(n)}$ never becomes zero is  the limit $\frak{p}_{b,s}^{(n)} \searrow  \frak{p}_{b,s}$.   The probabilities $\frak{p}_{b,s}^{(n)}$ satisfy the recursive relation 
$$ \frak{p}_{b,s}^{(n+1)}\,=\,  G_{b,s}\big( \frak{p}_{b,s}^{(n)}  \big) \hspace{1cm}\text{for}\hspace{1cm} G_{b,s}(x)\,:=\, \frac{1}{b}\Big[ 1\,-\,\big(1-x\big)^s  \Big]    $$
and have initial value $\frak{p}_{b,s}^{(0)}=1$. When $s$ is larger than $b$, the probability $ \frak{p}_{b,s}\in (0,1)$ is the unique attractive fixed point of the map $G_{b,s}:[0,1]\rightarrow [0,1]$.\vspace{.2cm}

\noindent Part (i): The variables $\mathbf{m}_n:=\big(\frac{b}{s}  \big)^{n}N_{p,q}^{(n)}$ form a nonzero, mean-one martingale with respect to the filtration, $\textup{F}_{n}^{(q)}$, generated by $[q]_n$.  Hence, there is an a.s.\ limit $\displaystyle \mathbf{m}_\infty = \lim_{n\rightarrow \infty} \mathbf{m}_n$.  The moment generating functions 
$\varphi_n(t):=\mathbb{E}[e^{t\mathbf{m}_n}   ]$ satisfy the recursive relation
 \begin{align}\label{MomGen}
    \varphi_{n+1}^{b,s}(t)\,=\,\frac{b-1}{b}\,+\,\frac{1}{b}\Big( \varphi_n^{b,s}\Big(\frac{b}{s}t\Big)\Big)^s   \,\hspace{1cm}\text{with}\hspace{1cm}\,   \varphi_0^{b,s}(t)\,=\,e^{t}  \,. 
   \end{align}
 The sequence $\varphi_{n}^{b,s}(t)$ converges pointwise to a nontrivial limit $\varphi_{\infty}^{b,s}(t)$, which is the moment generating function of $\mathbf{m}_\infty$, satisfying  $\varphi_{\infty}^{b,s}(t)=\frac{b-1}{b}+\frac{1}{b}\big( \varphi_\infty^{b,s}\big(\frac{b}{s}t\big)\big)^s $.  Note that the limit of $\varphi_{\infty}^{b,s}(t)$ as $t\rightarrow -\infty$ solves $x=\frac{b-1}{b}\,+\,\frac{1}{b}x^s$, and thus   $\mathbb{P}\big[ \mathbf{m}_\infty = 0\big]=1-\frak{p}_{b,s}$.

\vspace{.3cm}

\noindent Parts (ii) and (iii):   In the event that $N_{p,q}^{(n)}$ is zero for some $n\in \mathbb{N}$, $\mathcal{I}_{p,q}$ is  finite and $p(t)=q(t)\in V^{b,s}$ for $t\in\mathcal{I}_{p,q}$.  Conversely, I will show below that if  $N_{p,q}^{(n)}$ is never zero, then the set $\mathcal{I}_{p,q}$ is a.s.\ infinite  since its dimension-$h$ Hausdorff measure is infinite for any $0<h< \frak{h}$.  This would suffice to prove the proposition since the above remarks show that $N_{p,q}^{(n)}$ becomes zero for large enough $n\in \mathbb{N}$ with probability $1-\frak{p}_{b,s}$.\vspace{.2cm}

 I will split up the analysis between proving $\textup{dim}_{H}(I_{p,q})\leq \frak{h}$  and $\textup{dim}_{H}(I_{p,q})\geq \frak{h}$. To show that $\textup{dim}_H(I_{p,q})\leq \frak{h}$, I will argue that the Hausdorff measure $H_{\frak{h}}(\mathcal{I}_{p,q})$ is a.s.\ finite.   For a given $\delta>0$ pick $n$ with $(\frac{1}{s})^n<\delta$.  Since $\mathcal{I}_{p,q}\subset \mathcal{I}_{p,q}^{(n)}$ and $\mathcal{I}_{p,q}^{(n)}$ is covered by $ N_{p,q}^{(n)}$ intervals of length $1/s^n$,
$$ H_{\frak{h},\delta}(\mathcal{I}_{p,q})\,=\, \inf_{ \substack{ \mathcal{I}_{p,q}\subset \cup_k \mathcal{I}_k \\ |\mathcal{I}_k|\leq \delta  \\  }   } \sum_k  |\mathcal{I}_k|^\frak{h}\,\leq \, N_{p,q}^{(n)}\Big(\frac{1}{s}\Big)^{\frak{h}n}\,=\,\mathbf{m}_{n} \,. $$
Thus,
$$ H_{\frak{h}}(\mathcal{I}_{p,q})\,=\,\lim_{\delta\rightarrow 0} H_{\frak{h},\delta}(\mathcal{I}_{p,q})\,\leq \,\liminf_{n\rightarrow \infty} \,\mathbf{m}_n\,= \,\mathbf{m}_\infty\,.  $$
Therefore $H_{\frak{h}}(\mathcal{I}_{p,q})$ is a.s.\ finite and $\textup{dim}_{H}(I_{p,q})\leq \frak{h}$. \vspace{.2cm}

Next I will condition on the event that $N_{p,q}^{(n)}$ is not zero for any $n\in \mathbb{N}$ and show  that $\textup{dim}_{H}(I_{p,q})\geq \frak{h}$ a.s.  It suffices to  show that $H_{h}(\mathcal{I}_{p,q})>0$ for any $0<h<\frak{h}$.  Let $\mathcal{S}^{(n)}_{p,q}$ be the collection of intervals $[ \frac{k-1}{s^n}, \frac{k}{ s^n}] \subset I_{p,q}^{(n)}$ such that $[ \frac{k-1}{s^n}, \frac{k}{ s^n}]\cap I_{p,q}^{(N)}  $ is not finite for any $N>n$ (in other terms, ancestors of the interval do not go extinct). For a Borel set $A\subset [0,1]$, let $ \mathcal{C}(A)$ be the set of coverings of $A$ by elements in $\cup_{n=1}^{\infty}\mathcal{S}^{(n)}_{p,q}$.   Define the Hausdorff-like measure $\widetilde{H}_{h}$ as 
\begin{align}\label{Ding}
\widetilde{H}_{h}(A)= \lim_{n\rightarrow \infty}\widetilde{H}_{h,\frac{1}{s^n}}(A)\, \hspace{1cm}\text{for}\hspace{1cm} \widetilde{H}_{h,\delta}(A)\,=\, \inf_{ \substack{ \{ \mathcal{I}_k \}\in \mathcal{C}(A)     \\ |\mathcal{I}_k|\leq \delta    }   } \sum_k  |\mathcal{I}_k|^h\,.
 \end{align}
For any Borel $A\subset [0,1]$ we have that
  \begin{align}   \frac{1}{2s^{h}} \widetilde{H}_{h}(A)    \,  \leq \, H_{h}(A)\, \leq \,  \widetilde{H}_{h}(A) \,.
\end{align}
The second inequality above holds since  $\widetilde{H}_{h,1/s^n}(A)$ is defined as an infimum over a smaller collection of coverings than $H_{h,1/s^n}(A)$. The first inequality holds since any interval $I\subset [0,1]$ is covered by two adjacent intervals of the form $\big[\frac{k-1}{s^n},\frac{k}{s^n}   \big]$ for  $n:=\lfloor \log_{1/s}|I| \rfloor$.  Thus, if $\widetilde{H}_{h,1}(\mathcal{I}_{p,q})>0 $ holds a.s.\ then  $H_{h}(\mathcal{I}_{p,q})>0$ holds a.s.

  Let $\widetilde{N}_{p,q}^{(n)}$ be the number elements in $\mathcal{S}^{(n)}_{p,q}$.  Conditioned on the event  that $ N_{p,q}^{(n)}$ is never zero, $\widetilde{N}_{p,q}^{(n)}$ forms a Markov chain taking values in $\mathbb{N}$ with initial value $\widetilde{N}_{p,q}^{(0)}=1$ and satisfying the distributional equality
\begin{align*}
\widetilde{N}_{p,q}^{(n+1)}\,\stackrel{d}{=}\,\sum_{j=1}^{\widetilde{N}_{p,q}^{(n)}  }\mathbf{n}_j\,\hspace{.2cm}\text{for i.i.d.\ variables $ \mathbf{n}_j\in \{1,\ldots, s\} $  with }\hspace{.2cm} \mathbb{P}\big[  \mathbf{n}_j=\ell \big]\,=\,  \binom{s}{\ell}\frac{  \frak{p}_{b,s}^{\ell}(1-\frak{p}_{b,s})^{s-\ell} }{ 1-(1-\frak{p}_{b,s})^{s}   }\,.
\end{align*}
Fix some $0<h<\frak{h}$.  Define the variables 
\begin{align}
L_{p,q,n}\,:=\, \inf_{ \substack{ \{ \mathcal{I}_k \}\in \mathcal{C}(\mathcal{I}_{p,q})     \\ |\mathcal{I}_k|\geq \frac{1}{s^n}  \\  }   } \sum_k  |\mathcal{I}_k|^h \,,
 \end{align}
which have the a.s.\ convergence $ L_{p,q,n}\searrow L_{p,q,\infty}:=\widetilde{H}_{h,1}(\mathcal{I}_{p,q})$.
The hierarchical symmetry of the model implies that the $L_{p,q,n}$'s satisfy the distributional recursion relation
\begin{align}\label{Ys} L_{p,q, n+1}\,\stackrel{d}{=}\, \min\Bigg(1,\,  \sum_{j=1}^{\mathbf{n}} \Big(\frac{1}{s}\Big)^{h}L_{p,q, n}^{(j)}  \Bigg)  \, ,
 \end{align}
where the $ L_{p,q, n}^{(j)}$'s are independent copies of $L_{p,q, n}$ and $\mathbf{n}\in \{1,\ldots, s\}$  is independent of the  $ L_{p,q, n}^{(j)}$'s with $
\mathbb{P}\big[  \mathbf{n}=\ell \big]\,=\,  \binom{s}{\ell}\frac{  \frak{p}_{b,s}^{\ell}(1-\frak{p}_{b,s})^{s-\ell} }{ 1-(1-\frak{p}_{b,s})^{s}   } $.
The distribution of $ L_{p,q, \infty}$ is a fixed point of~(\ref{Ys}).  The probability $x=\mathbb{P}\big[  L_{p,q, \infty}=0  \big]$ satisfies 
$$x\, =\, \frac{ (x\frak{p}_{b,s} +   1-\frak{p}_{b,s})^{s} -(1-\frak{p}_{b,s})^{s}   }{  1-(1-\frak{p}_{b,s})^{s} }   \, ,$$
which has solutions only for $x=0$ and $x=1$.  However, $x=1$ is not possible since a.s.\ convergence $L_{p,q, n}\searrow 0$ as $n\rightarrow \infty$ contradicts~(\ref{Ys}).  To see the rough idea for this, notice that if $0<L_{p,q, n}\ll 1$ with high probability when $n\gg 1$ then the expectation of~(\ref{Ys}) yields 
$$\mathbb{E}\big[L_{p,q, n+1}\big] \,\approx \,  \mathbb{E}\Bigg[  \sum_{j=1}^{\mathbf{n}} \Big(\frac{1}{s}\Big)^{h}L_{p,q, n}^{(j)}  \Bigg] \,=\, \Big( \frac{1}{s} \Big)^{h} \,\mathbb{E}[\mathbf{n}] \,\mathbb{E}\big[L_{p,q, n}\big]\,=\, s^{\frak{h}-h} \,\mathbb{E}\big[L_{p,q, n}\big] \,   $$
because $\mathbb{E}[\mathbf{n}]=\frac{s}{b}= s^{\frak{h}} $.   The above shows that the expectations of $\mathbb{E}\big[L_{p,q, n}\big]$ will contract away from $0$ since $\frak{h}-h>0$.

\end{proof}

\subsection{The compact operator $Y$}

In this section I will prove Propositions~\ref{PropYProp} and~\ref{PropYs}. \vspace{.2cm}

\begin{definition}\label{DefEigen} For $\ell\in \{1,\ldots, b\}$, let  $v^{(\ell)}=(v_1^{(\ell)},\ldots,v_{b}^{(\ell)}) $ be orthonormal vectors in $\R^b$ where $v^{(1)}=\frac{1}{\sqrt{b}}(1,\ldots, 1)$.   Let $\mathbf{p}_1,\ldots, \mathbf{p}_b$ be an enumeration of the elements in $\Gamma_1^{b,s}$, i.e., the branches of $D_1^{b,s}$.

\begin{itemize}

\item  Define  $f^{(\ell)}\in \mathcal{H}$ and  $\widehat{f}^{(\ell)}\in  L^2(\Gamma^{b,s},\mu)  $  for $\ell\in\{1,\ldots, b\}$ as 
$$f^{(\ell)}(x)\,= \sqrt{b}\sum_{i=1}^b  v^{(\ell)}_i\chi_{ \cup_{k=1}^s [\mathbf{p}_i ](k)  } (x)\hspace{.8cm}\text{and}\hspace{.8cm}\widehat{f}^{(\ell)}(p)\,= \sqrt{b}\sum_{i=1}^b  v^{(\ell)}_i\chi_{ \mathbf{p}_i   } (p) \,.  $$

\item For $(\mathbf{e},\ell)\in \cup_{n=0}^\infty E^{b,s}_{n} \times \{1,\ldots,b\}$, define $f_{(\mathbf{e},\ell)}\in \mathcal{H} $ as
\begin{align*}
f_{(\mathbf{e},\ell)}(x)\,=\, (sb)^{\frac{n}{2}}\chi_\mathbf{e}(x)   f^{(\ell)}( x_{\mathbf{e}} )\, ,
\end{align*}
where for $x\in \mathbf{e}$ the point $x_{\mathbf{e}}\in D^{b,s}$ refers to the position of $x$ in the shrunken copy of $D^{b,s}$ corresponding to $\mathbf{e}$.

\item For $(\mathbf{e},\ell)\in \cup_{n=0}^\infty E^{b,s}_{n} \times \{1,\ldots,b\}$, define $\widehat{f}_{(\mathbf{e},\ell)}\in L^2(\Gamma^{b,s},\mu) $ as
\begin{align*}
\widehat{f}_{(\mathbf{e},\ell)}(p)\,=\, b^{\frac{n}{2}}\chi_{\mathbf{e}\cap\textup{Range}(p)\neq \emptyset  }   \widehat{f}^{(\ell)}( p_{\mathbf{e}} )\,,
\end{align*}
where if $\mathbf{e}\cap\textup{Range}(p)\neq \emptyset $ the path $p_{\mathbf{e}}\in \Gamma^{b,s}$ refers to a magnification of the portion of the path $p$ in the shrunken copy of $D^{b,s}$ corresponding to $\mathbf{e}$.

\end{itemize}

\end{definition}

\begin{proof}[Proof of Proposition~\ref{PropYProp}]  It suffices to show that the operator $Y:\mathcal{H}\rightarrow L^2(\Gamma^{b,s},\mu)$ has the form
\begin{align}\label{RepY}
Y\,=\, \big|  \widehat{f}_{(D^{b,s},1)} \big\rangle \big\langle f_{(D^{b,s},1)} \big|  \,+\,  \sum_{k=0}^\infty \sum_{\substack{ \mathbf{e}\in E^{b,s}_{k}   \\  \ell\in \{2,\ldots,b\}   }  } s^{-\frac{k}{2}}\big|  \widehat{f}_{(\mathbf{e},\ell)} \big\rangle \big\langle f_{(\mathbf{e},\ell)} \big| \,. 
\end{align}
 Clearly $Y $ maps   $ f_{(D^{b,s},1)}=1_{D^{b,s}  }$ to $ \widehat{f}_{(D^{b,s},1)}=1_{\Gamma^{b,s}}$.  Pick $\mathbf{e}\in E^{b,s}_{k}$ and $ \ell\in \{2,\ldots,b\} $.  For $p\in \Gamma^{b,s}$,
\begin{align*}  
 \big(Yf_{(\mathbf{e},\ell)}\big)(p) &\,=\,\int_0^1 f_{(\mathbf{e},\ell)}\big( p(r)  \big)dr \\ &\,=\, \int_0^1 (sb)^{\frac{k}{2}}\chi_{\mathbf{e}}\big(p(r)   \big)f^{(\ell)}\big( (p(r))_{\mathbf{e}}  \big)dr \\  & \,=\, \frac{1}{s^k} (sb)^{\frac{k}{2}}\chi_{\mathbf{e}\cap \textup{Range}(p)  \neq \emptyset}f^{(\ell)}\big( p_{\mathbf{e}}  \big)      \,=\,s^{-\frac{k}{2}}\widehat{f}_{(\mathbf{e},\ell)}(p) \,. 
 \end{align*}
The third equality holds since the path $p(r)$ is in $\mathbf{e}$ for a time interval of length $1/s^k$ in the event that $\mathbf{e}\cap \textup{Range}(p)$ is nonempty. 

The orthogonal complement in $\mathcal{H}$ of the space  spanned by the  vectors $f_{(D^{b,s},1)}$ and $ f_{(\mathbf{e},\ell)}$ with  $(\mathbf{e},\ell)\in E^{b,s}_{k}\times \{2,\ldots,b\}$ is comprised of all vectors $\psi$ for which $0=\int_0^1 \psi(p(r))dr$ for all $p\in \Gamma^{b,s}$, which is, by definition, the null space of $Y$.

\end{proof}

\begin{proof}[Proof of Proposition~\ref{PropYs}]

\noindent Part (i): The conditional expectation $\textup{E}\big[\,\cdot\, \big| \,\textup{F}_n\big]$  satisfies
$$\textup{E}\big[\widehat{f}_{(\mathbf{e},\ell)}\,\big|\,\textup{F}_n   \big]\,=\, \begin{cases}  \widehat{f}_{(\mathbf{e},\ell)}   &  \mathbf{e}\in \bigcup_{k=0}^{n-1} E^{b,s}_k\,,   \\  0  & \mathbf{e}\in \bigcup_{k=n}^{\infty} E^{b,s}_k \,\text{and}\,\ell \in \{2,\ldots,b\}    \,.\end{cases}  $$
The result then follows from the form~(\ref{RepY}) of $Y$ since $Y^{(n)}$ has the form
\begin{align*}
Y^{(n)}\,=\, \big|\widehat{f}_{(D^{b,s},1)} \big\rangle \big\langle f_{(D^{b,s},1)} \big|  \,+\,  \sum_{k=0}^{n-1} \sum_{\substack{ \mathbf{e}\in E^{b,s}_{k}   \\  \ell\in \{2,\ldots,b\}   }  } s^{-\frac{k}{2}}\big|  \widehat{f}_{(\mathbf{e},\ell)} \big\rangle \big\langle f_{(\mathbf{e},\ell)} \big|\,=\,Y\mathbf{P}_n \,. 
\end{align*}

\vspace{.3cm}

\noindent Part (ii):  As a consequence of (i), the operator norm of the difference between  $Y^{(n)}$ and $Y$ has the form $\|Y^{(n)}-Y\|_\infty =s^{-n/2}  $.\vspace{.5cm}

\noindent Part (iii):  The operator $Y^{(n)}(Y^{(n)})^*$ can be written in the form
\begin{align*}
Y^{(n)}(Y^{(n)})^*\,=\, \big|\widehat{f}_{(D^{b,s},1)} \big\rangle \big\langle \widehat{f}_{(D^{b,s},1)} \big|  \,+\,  \sum_{k=0}^{n-1} \sum_{\substack{ \mathbf{e}\in E^{b,s}_{k}   \\  \ell\in \{2,\ldots,b\}   }  } s^{-k}\big|  \widehat{f}_{(\mathbf{e},\ell)} \big\rangle \big\langle \widehat{f}_{(\mathbf{e},\ell)} \big| \,. 
\end{align*}
It follows that the Hilbert-Schmidt norm of the difference between $Y^{(n)}(Y^{(n)})^*$ and $Y Y^*$ is $(\frac{b}{s})^{ n/2}\sqrt{ s\frac{b-1   }{ s-b  }  }$.

Next I show that $Y Y^*$ has integral kernel $T(p,q)$.  The map  $(Y^{(n)})^*:L^{2}\big(\Gamma^{b,s},\mu   \big)\rightarrow \mathcal{H} $ sends $f\in  L^{2}\big(\Gamma^{b,s},\mu   \big)$ to 
$$  (Y^{(n)})^* f\,=\, \int_{\Gamma^{b,s}} f(p) Y^{(n)}_p\,    dp \,=\,  b^n\int_{\Gamma^{b,s}}  f(p)\chi_{T_p^{(n)}}   dp    \,. $$
Thus, $Y^{(n)}(Y^{(n)})^*:L^{2}\big(\Gamma^{b,s},\mu   \big)\rightarrow L^{2}\big(\Gamma^{b,s},\mu   \big)  $ has kernel $K^{(n)}(p,q)$
$$ K^{(n)}(p,q)\,=\, \big\langle Y^{(n)}_p, Y^{(n)}_q  \big\rangle \,=\, \frac{b^n}{s^n}\sum_{k=1}^{s^{n}}\mathbf{1}_{ [p]_n(k)= [q]_n(k)  } \,:=\,\mathbf{m}_n     \,. $$
However, $\mathbf{m}_n\equiv \mathbf{m}_n(p,q)$ converges to $\mathbf{m}_\infty(p,q) = T_{p,q}$ in  $L^{2}\big(\Gamma^{b,s}\times \Gamma^{b,s},\mu\times \mu\big)$ by part (i) of Proposition~\ref{PropPathInter}.

\vspace{.5cm}

\noindent Part (iv):  The vectors $Y^{(n)}_p\in \mathcal{H}$ satisfy $Y^{(k)}_p= \mathbf{P}_k Y^{(n)}_p$ for $k\leq n$.  Thus,
\begin{align*}
Y^{(n)}_p\,-\,Y^{(n-1)}_p\,=\, (\mathbf{P}_{n}-\mathbf{P}_{n-1} ) Y^{(n)}_p\,\in\, \mathcal{H}_n\,.
\end{align*}

\end{proof}

\subsection{Existence and uniqueness of the GMC measure for the CDRP}

\subsubsection{GMC martingale}

\begin{proposition}\label{PropYField} Define  $\mathcal{F}_{n}$ to be the  $\sigma$-algebra on $\Omega$ generated by variables $\langle \mathbf{W}, \psi\rangle$ for $\psi\in \oplus_{k=0}^n \mathcal{H}_k $ where $\mathcal{H}_{k}$ is defined as in part (ii) of Proposition~\ref{PropYProp}.
\begin{enumerate}[(i).]
\item  The sequence of random variables $\big\{ e^{\beta \langle \mathbf{W}, Y_p^{(n)} \rangle -\frac{\beta^2}{2}\| Y_p^{(n)} \|_{\mathcal{H}}^2   } \big\}_{n\in \mathbb{N}}$ forms a mean-one martingale with respect to $\mathcal{F}_{n}$.

\item  For any Borel set $A\subset \Gamma^{b,s}$, the sequence of random variables  $\displaystyle \Big\{  \int_{A}e^{\beta \langle \mathbf{W}, Y_p^{(n)} \rangle -\frac{\beta^2}{2}\| Y_p^{(n)} \|_{\mathcal{H}}^2   } \mu(dp) \Big\}_{n\in \mathbb{N}}$ is a mean-$\mu(A)$, square-integrable martingale with respect to $\mathcal{F}_{n}$ that converges a.s.\ to a nonzero limit.

\end{enumerate}

\end{proposition}

\begin{proof}Part (i): If $N>n$, then  $\langle \mathbf{W}, Y_p^{(N)}- Y_p^{(n)} \rangle $ and $\langle \mathbf{W}, Y_p^{(n)} \rangle $ are independent normal random variables since $Y_p^{(N)}- Y_p^{(n)}$ and $Y_p^{(n)}$ are orthogonal elements in $\mathcal{H}$ by  part (iv) of Proposition~\ref{PropYs}.  Thus, we can write
\begin{align*}
e^{\beta \langle \mathbf{W}, Y_p^{(N)} \rangle -\frac{\beta^2}{2}\| Y_p^{(N)} \|_{\mathcal{H}}^2   } \,=\, e^{\beta \langle \mathbf{W}, Y_p^{(N)}- Y_p^{(n)} \rangle -\frac{\beta^2}{2}\| Y_p^{(N)}-Y_p^{(n)} \|_{\mathcal{H}}^2   } \, e^{\beta \langle \mathbf{W}, Y_p^{(n)} \rangle -\frac{\beta^2}{2}\| Y_p^{(n)} \|_{\mathcal{H}}^2   }   \,.
\end{align*}
The conditional expectation with respect to $\mathcal{F}_n$ is $   e^{\beta \langle \mathbf{W}, Y_p^{(n)} \rangle -\frac{\beta^2}{2}\| Y_p^{(n)} \|_{\mathcal{H}}^2   }  $.
  
\vspace{.3cm}
\noindent Part (ii): The sequence $\displaystyle \Big\{  \int_{A}e^{\beta \langle \mathbf{W}, Y_p^{(n)} \rangle -\frac{\beta^2}{2}\| Y_p^{(n)} \|_{\mathcal{H}}^2   } \mu(dp) \Big\}_{n\in \mathbb{N}}$ is a mean-$\mu(A)$ martingale.  We can write
\begin{align*}
\int_{A}e^{\beta \langle \mathbf{W}, Y_p^{(n)} \rangle -\frac{\beta^2}{2}\| Y_p^{(n)} \|_{\mathcal{H}}^2   }\mu(dp)\,=\,\int_{A} \prod_{k=1}^{s^n}e^{ \beta \langle \mathbf{W}, b^n\chi_{[p]_n(k)  } \rangle -\frac{\beta^2}{2}\| b^n\chi_{[p]_n(k) } \|_{\mathcal{H}}^2  }  \mu(dp)  
\end{align*}
Since the $ \langle \mathbf{W}, b^n\chi_{[p]_n(k)  } \rangle$'s are independent, centered normal variables with variance $(b/s)^n$, the second moment of the above is equal to 
\begin{align}\label{Cross}
\int_{A\times A}  e^{\beta^2(\frac{b}{s})^n N_{p,q}^{(n)}   }\mu(dp)\mu(dq)\,,
\end{align}  
where $N_{p,q}^{(n)}  $ is the number of bonds shared by the coarse-grained paths $[p]_n$ and $[q]_n$. By part (i) of Proposition~\ref{PropPathInter}, the sequence  $(\frac{b}{s})^n N_{p,q}^{(n)}$ converges  $\mu\times \mu$-a.e.\  as $n\rightarrow \infty$  to the intersection time $T_{p,q}$, and when $A=\Gamma^{b,s}$ the expression~(\ref{Cross}) converges to the moment generating function, $\mathbb{E}[\textup{exp}(\beta^2T_{p,q})]$. It follows that~(\ref{Cross}) converges to  $\int_{A\times A}  e^{\beta^2T_{p,q}  }\mu(dp)\mu(dq)$ for an arbitrary Borel set $A\in \mathcal{B}_\Gamma$.

\end{proof}

\subsubsection{Martingale limit construction of the GMC measure}

\begin{proof}[Proof of Propositions~\ref{ThmExist} and~\ref{PropGMCApprox}] Recall that $M^{(n)}_\beta$  is defined as the GMC measure on $(\Gamma^{b,s},\mu)$ over the  finite-dimensional field $(\mathbf{W}, \beta Y^{(n)})$:
\begin{align}\label{DefMsubN}
M^{(n)}_\beta(A)\,=\,  \int_{A}e^{\beta \langle \mathbf{W}, Y_p^{(n)} \rangle -\frac{\beta^2}{2}\| Y_p^{(n)} \|_{\mathcal{H}}^2   } \mu(dp)\,,\hspace{1.5cm} A \in  \mathcal{B}_\Gamma \,.
\end{align}
The space $\mathcal{H}_Y=\oplus_{k=0}^\infty \mathcal{H}_k  $ is the orthogonal complement of the null space of $Y$.  Define $\mathcal{F}_\infty$ as the $\sigma$-algebra generated by the variables  $\langle \mathbf{W}, \phi\rangle $ for $\phi \in \mathcal{H}_Y$.\vspace{.2cm}

\noindent \textbf{Existence:} Let $D$ be a  countable subcollection of continuous functions on $\Gamma^{b,s}$ that are dense with respect to the supremum  norm (where, recall, $d_\Gamma$ is the underlying metric on $\Gamma^{b,s}$).   For each $\psi\in D$, the sequence $\big\{\int_{\Gamma^{b,s}}\psi(p) M^{(n)}_\beta(dp)\big\}_{n\in \mathbb{N}}$ is a  martingale w.r.t.\ the filtration $\mathcal{F}_n$ having uniformly bounded second moments as a consequence of Proposition~\ref{PropYField}.  Consequently,  $\big\{\int_{\Gamma^{b,s}}\psi(p) M^{(n)}_\beta(dp)\big\}_{n\in \mathbb{N}}$ converges a.s.\ to a limit in $L^{2}(\Omega, \mathcal{F},\mathbb{P})$.  Thus, the measures $  M^{(n)}_\beta$  a.s.\ converge vaguely to a limit measure $M_\beta$ adapted to the field $\mathbf{W}$, i.e., $M_\beta\equiv M_\beta(\mathbf{W})$.

Properties (I)-(III) follow easily from this construction.  For instance,  to verify property (III) fix some $\phi\in \bigcup_{N=1}^{\infty} \oplus_{k=0}^N \mathcal{H}_k$. Notice that since $  M^{(n)}_\beta(\mathbf{W},dp)$ a.s.\ converges vaguely to $M_\beta(\mathbf{W},dp)$,   I have the a.s.\ equality  
\begin{align}
M_\beta(\mathbf{W}+\phi,dp)  \, =\,& \lim_{n\rightarrow \infty} M_\beta^{(n)}(\mathbf{W}+\phi,dp)\nonumber \\
\, =\,& \lim_{n\rightarrow \infty}e^{\beta\langle Y^{(n)}_p,\phi\rangle}  M_\beta^{(n)}(\mathbf{W},dp)\,. \nonumber 
\intertext{However, since $\phi \in  \oplus_{k=0}^N \mathcal{H}_k $ for some $N$,  I have  $\langle Y^{(n)}_p,\phi\rangle=\langle Y_p,\phi\rangle$ when $n\geq N$, and thus,}
\, =\,& e^{\beta\langle Y_p,\phi\rangle} \lim_{n\rightarrow \infty}  M_\beta^{(n)}(\mathbf{W},dp) \nonumber \\
\, =\,& e^{\beta\langle Y_p,\phi\rangle}  M_\beta(\mathbf{W},dp) \,.
\end{align}
The space $\bigcup_{N=1}^{\infty}\oplus_{k=0}^N \mathcal{H}_k$ is dense in the complement of the null space of $Y$, and  it follows that the above extends to all elements in $\mathcal{H}$.\vspace{.5cm}

\noindent \textbf{Uniqueness:} Next I argue that  $M_\beta$ is the unique  GMC measure over the field $(\mathbf{W},\beta Y)$.  I will reduce the uniqueness of $M_\beta$ to  the uniqueness of the  the GMC measures $M_\beta^{(n)}$  over the finite-dimensional fields $(\mathbf{W},\beta Y^{(n)})$.  Let $\widetilde{M}_\beta$  be a random measure satisfying (I)-(III), and define $\widetilde{M}_\beta^{(n)}$ as the conditional expectation of $\widetilde{M}_\beta  $ w.r.t.\ $\mathcal{F}_n $
\begin{align}\label{MTilde}
 \widetilde{M}_\beta^{(n)}\big(\mathbf{W}^{(n)},A\big)\,=\, \mathbb{E}\big[ \widetilde{M}_\beta\big(\mathbf{W},A\big)\,\big|\,\mathcal{F}_n  \big]\,,  
 \end{align}
where $\mathbf{W}^{(n)} $ refers to the finite-dimensional field of variables $\langle \mathbf{W},\psi\rangle $ for $\psi\in \oplus_{k=0}^n\mathcal{H}_k$.  Since $\mathcal{H}=\oplus_{k=0}^\infty\mathcal{H}_k$, the random variable $\widetilde{M}_\beta\big(\mathbf{W},A\big)$ is recovered as the a.s.\ limit
\begin{align}\label{MRecover}
\widetilde{M}_\beta\big(\mathbf{W},A\big)\,=\,\lim_{n\rightarrow \infty}\mathbb{E}\big[ \widetilde{M}_\beta\big(\mathbf{W},A\big)\,\big|\,\mathcal{F}_n  \big]\,.
\end{align}

 Obviously~(\ref{MTilde}) implies that $\mathbb{E}[ \widetilde{M}_\beta^{(n)}  ]=\mathbb{E}[ \widetilde{M}_\beta  ]=\mu $.  For $\phi\in  \oplus_{k=0}^n\mathcal{H}_k$,
$$ \widetilde{M}_\beta^{(n)}\big(\mathbf{W}^{(n)}+\phi,A\big)\,=\, \mathbb{E}\big[ \widetilde{M}_\beta\big(\mathbf{W}+\phi,A\big)\,\big|\,\mathcal{F}_n  \big]\,=\,e^{\beta\langle Y_p,\phi\rangle   } \mathbb{E}\big[ \widetilde{M}_\beta\big(\mathbf{W},A\big)\,\big|\,\mathcal{F}_n  \big]\,=\,e^{\beta\langle Y_p,\phi\rangle   }\widetilde{M}_\beta^{(n)}\big(\mathbf{W}^{(n)},A\big) \,. $$
If $\widetilde{M}_\beta^{(n)}\equiv  \widetilde{M}_\beta^{(n)} (\mathbf{W})$ is viewed as a function of the entire field $\mathbf{W}$, then 
\begin{align*}
 \widetilde{M}_\beta^{(n)}(\mathbf{W}+\phi,dp)\,=\,  e^{\beta\langle Y_p^{(n)},\phi\rangle   } \widetilde{M}_\beta^{(n)}(\mathbf{W},dp)
\end{align*}
since $ Y_p^{(n)}=P_nY_p$ for the projection $P_n:\mathcal{H}\rightarrow \oplus_{k=0}^n\mathcal{H}_k $. Thus,  $\widetilde{M}_\beta^{(n)}$ is a GMC  measure over the trivial field $(\mathbf{W},\beta Y^{(n)})$ and must be equal to $M_\beta^{(n)}$. From the construction analysis above, $M_\beta^{(n)}$ a.s.\ converges vaguely to a GMC measure $M_\beta$ over the field $(\mathbf{W},\beta Y)$.  From~(\ref{MRecover}) it follows that $M_\beta=\widetilde{M}_\beta$.

\end{proof}

\subsection{Properties of the GMC measure}

In this section I will prove  Theorem~\ref{ThmProp}.

\begin{proof}[Proof of Theorem~\ref{ThmProp}]  Part (i): First I will show that for $\beta>0$ the GMC measure $M_\beta$ is  mutually singular to $\mu$ with probability $0$ or $1$. Let $ M^{(c)}_\beta$ and $M^{(s)}_\beta$ be the continuous and singular components in the  Lebesgue decomposition of $M_\beta$ w.r.t.\ $\mu$.  Since the white noise underlying $M_\beta$ has a uniform covariance measure $\nu$ over $D^{b,s}$, there is statistical uniformity in how the random measure $M_\beta$ weighs  the space of directed paths, $\Gamma^{b,s}$.  This, in particular, implies that the components $M^{(c)}_\beta$ and $M^{(s)}_\beta $ must have uniform expectations, and hence    $\mathbb{E}\big[ M^{(c)}_\beta\big]=\lambda \mu$ and $\mathbb{E}\big[ M^{(s)}_\beta\big]=(1-\lambda) \mu$ for some value $\lambda\in[0,1]$.  However, if $\lambda\in (0,1)$ then the random measures $\lambda^{-1}M^{(c)}_\beta$ and $(1-\lambda)^{-1}M^{(s)}_\beta$ must both be GMC measures satisfying properties (I)-(III) of Proposition~\ref{ThmExist}.  This, however, contradicts the  uniqueness of that GMC measure.  Therefore, $\lambda\in \{0,1\}$.  Suppose to reach a contradiction that $\lambda=1$, i.e., $M_\beta=M^{(c)}_\beta$.  Let $G_\beta(\mathbf{W},p)$ be the Radon-Nikodym derivative of    $M_\beta$ with respect to $\mu$.  For almost every $p$, $G_\beta(\mathbf{W},p)$ is a random variable with finite second moment and for all $\phi\in \mathcal{H}$
\begin{align}
G_\beta(\mathbf{W}+\phi ,p)\,=\,e^{\beta\langle \phi, Y(p)\rangle  }G_\beta(\mathbf{W},p)\,.
\end{align}
However, the above is only possible if $Y(p)\in \mathcal{H}$, which is a contradiction.

 \vspace{.3cm}   
      
\noindent Part (ii):  The intersection time, $T_{p,q}$, of two  random paths $p,q\in \Gamma^{b,s}$ chosen independently according to the measure $M_\beta$ has moment generating function $ \int_{\Gamma^{b,s}\times \Gamma^{b,s}}e^{\alpha T_{p,q}  } M_\beta(dp) M_\beta(dq)$, which has expectation
\begin{align}\label{Char}
\mathbb{E}\bigg[ \int_{\Gamma^{b,s}\times \Gamma^{b,s}}e^{\alpha T_{p,q}  } M_\beta(dp) M_\beta(dq)  \bigg]   \,=\,&\int_{\Gamma^{b,s}\times \Gamma^{b,s}}e^{(\alpha+\frac{1}{2}\beta^2) T_{p,q}  }\mu(dp)\mu(dq)\nonumber  \\ \,=&\,\varphi^{b,s}_{\infty}\big( \alpha+\beta^2/2  \big)\,<\,\infty \,.
\end{align}   
It follows that the set of pairs $(p,q)$ for which the intersection set $I_{p,q}$ has Hausdorff dimension $> \frak{h}$,  and thus for which $T_{p,q}=\infty$, is a.s.\ of measure zero with respect to $M_\beta\times M_\beta$.  The set of pairs $(p,q)$ for which the intersection set $I_{p,q}$ has Hausdorff dimension $< \frak{h}$ satisfy $T_{p,q}=0  $, and
\begin{align*}
\lim_{\alpha\rightarrow -\infty} \int_{\Gamma^{b,s}\times \Gamma^{b,s}} e^{\alpha T_{p,q}  } M_\beta(dp) M_\beta(dq)\, = \, M_\beta \times M_\beta\Big(\big\{(p,q)\,\big|\, T_{p,q}=0   \big\}\Big) \,.
\end{align*}
Taking the expectation 
\begin{align}\label{Dample}
\mathbb{E}\Big[ M_\beta \times M_\beta\Big(\big\{(p,q)\,\big|\, T_{p,q}=0   \big\}\Big)  \Big]\,=\,\lim_{\gamma\rightarrow -\infty}\varphi^{b,s}_\infty(\gamma)\,=\,\frak{p}_{b,s}\,.
\end{align}
However, $\big\{(p,q)\,\big|\, T_{p,q}=0   \big\}$ contains the set of pairs such that the intersection set $I_{p,q}$ is finite and 
\begin{align}\label{Dinkle}
\mathbb{E}\Big[ M_\beta \times M_\beta\Big(\big\{(p,q)\,\big|\, |I_{p,q}|<\infty   \big\}\Big)  \Big]\,=\, \mu \times \mu\Big(\big\{(p,q)\,\big|\, |I_{p,q}|<\infty   \big\}\Big)  \,=\,\frak{p}_{b,s}\,.
\end{align}
The second equality above holds by part (ii) of Proposition~\ref{PropPathInter}.   To see the first equality in~(\ref{Dinkle}), notice that $\big\{(p,q)\,\big|\, |I_{p,q}|<\infty   \big\}$ can be written  as the $n\rightarrow \infty$ limit
$$ \bigcup_{\substack{\mathbf{p},\mathbf{q}\in \Gamma^{b,s}_n \\ \forall(k)\,\mathbf{p}(k)\neq \mathbf{q}(k)  } } \mathbf{p}\times \mathbf{q}  \,\, \nearrow\,\,  \big\{(p,q)\,\big|\, |I_{p,q}|<\infty   \big\}  \,. $$   
Moreover, the random measure $M_\beta$ is independent over sets $\mathbf{p},\mathbf{q}\in \Gamma^{b,s}_n$ for which $\mathbf{p}(k)\neq \mathbf{q}(k) $ for $1\leq k\leq s^n$.

It follows from~(\ref{Dample}) and~(\ref{Dinkle}) that 
\begin{align}
\mathbb{E}\Big[ M_\beta \times M_\beta\Big(\big\{(p,q)\,\big|\, T_{p,q}=0 \text{ and } |\mathcal{I}_{p,q}|=\infty   \big\}\Big)  \Big]\,=\,0\,.
\end{align}
Therefore, $ M_\beta \times M_\beta$ is supported on pairs $(p,q)$ for which either $I_{p,q}$ is finite or has Hausdorff dimension $\frak{h}$.

\vspace{.3cm}

\noindent Part (iii): Let $M_{\beta }^{(i,j)}$ be measurable with respect to independent copies  $ \mathbf{W}^{(i,j)}$ of the field $\mathbf{W}$, and 
define the field $\widetilde{\mathbf{W}}$ such that 
\begin{align*}
\langle \widetilde{\mathbf{W}},\phi\rangle\,:=\,\frac{1}{\sqrt{bs}}\sum_{i=1}^{b}\sum_{j=1}^{s}\langle \mathbf{W}^{(i,j)},\phi^{(i,j)}\rangle
\hspace{1cm}\text{for}\hspace{1cm} \phi \in \widetilde{\mathcal{H}}=\bigoplus_{i=1}^{b}\bigoplus_{j=1}^{s}\mathcal{H}^{(i,j)}\,,
\end{align*}
where $\mathcal{H}^{(i,j)}$ are copies of the Hilbert space $\mathcal{H}=L^2\big(D^{b,s},\mathcal{B}_D,\nu   \big)$.  Define $( \widetilde{\Gamma}^{b,s}, \widetilde{\mu})$ for 
\begin{align}
\widetilde{\Gamma}^{b,s}\,:=\, \bigcup_{i=1}^{b}\bigtimes_{j=1}^s  \Gamma^{b,s}_{i,j} \hspace{1cm}\text{and}\hspace{1cm}  \widetilde{\mu}\,:= \,\frac{1}{b}\sum_{i=1}^b  \prod_{j=1}^s  \mu^{(i,j)}\,.
\end{align}
Finally, define $\widetilde{Y}:\mathcal{H}\longrightarrow L^{2}\big(  \widetilde{\Gamma}^{b,s}, \widetilde{\mu} \big)   $ as
\begin{align*}
\langle \widetilde{Y},\phi\rangle(p)\,=\,\frac{1}{s}\sum_{i=1}^{b}\chi\bigg(p\in \bigtimes_{j=1}^s  \Gamma^{b,s}_{i,j} \bigg)\sum_{j=1}^{s}\big\langle Y^{(i,j)},\phi^{(i,j)}\big\rangle\big(p^{(j)}\big)\,.
\end{align*}
In the above $p_j\in \Gamma_{i,j}^{b,s}$ are components of $p\in \bigtimes_{j=1}^s  \Gamma^{b,s}_{i,j} $.

The computation below shows that $\widetilde{M}= \frac{1}{b}\sum_{i=1}^b  \prod_{j=1}^s  M_{\beta }^{(i,j)}$ defines a GMC measure over the field $(\widetilde{\mathbf{W}},\sqrt{\frac{s}{b}}\beta \widetilde{Y})$:
\begin{align*}
\widetilde{M}\big(\widetilde{\mathbf{W}}+\phi,dp\big)\,=\,& \frac{1}{b}\sum_{i=1}^b  \chi\bigg(p\in \bigtimes_{j=1}^s  \Gamma^{b,s}_{i,j} \bigg)\prod_{j=1}^s  M_{\beta }^{(i,j)}\bigg(\mathbf{W}^{(i,j)}+\frac{\phi^{(i,j)}}{\sqrt{bs}},dp^{(j)}\bigg)\\
\,=\,& \frac{1}{b}\sum_{i=1}^b  \chi\bigg(p\in \bigtimes_{j=1}^s  \Gamma^{b,s}_{i,j} \bigg)e^{\beta\frac{1}{\sqrt{bs}}\sum_{j=1}^{s}\big\langle Y^{(i,j)},\phi^{(i,j)}\big\rangle\big(p^{(j)}\big)      }\prod_{j=1}^s  M_{\beta }^{(i,j)}\big(\mathbf{W}^{(i,j)},dp^{(j)}\big)\\
\,=\,& \frac{1}{b}\sum_{i=1}^b  \chi\bigg(p\in \bigtimes_{j=1}^s  \Gamma^{b,s}_{i,j} \bigg)e^{\sqrt{\frac{s}{b}}\beta\langle \widetilde{Y},\phi\rangle(p)  }\prod_{j=1}^s  M_{\beta }^{(i,j)}\big(\mathbf{W}^{(i,j)},dp^{(j)}\big)\\
\,=\,&e^{\sqrt{\frac{s}{b}}\beta\langle \widetilde{Y},\phi\rangle(p)  }\widetilde{M}\big(\widetilde{\mathbf{W}},dp\big)
\end{align*}
Therefore, the GMC measure $\widetilde{M}$ is equal in law to $M_{\sqrt{\frac{s}{b}}\beta}$.

\end{proof}

\subsection{Strong disorder behavior as $\beta\rightarrow \infty$      }

In this section I will prove Theorem~\ref{ThmDisorder}.  The proof below that $M_{\beta}\big(\mathbf{W},\Gamma^{b,s}\big)$ converges in probability to zero is a straightforward adaption of the argument of Lacoin and Moreno for discrete polymers on diamond lattices in~\cite{Lacoin}.

\begin{proof}[Proof of part (i) of Theorem~\ref{ThmDisorder}] It suffices to show that the fractional moment  $\mathbb{E}\big[\sqrt{M_{\beta}(\mathbf{W},\Gamma^{b,s})}  \big]$ converges to zero as $n\rightarrow \infty$.

Let $h:D^{b,s}  \rightarrow \R     $ be the constant function $h(x)=\lambda$ for some $\lambda\in \R$,  and  $\widehat{\mathbb{P}}_\lambda$ be the measure on $\mathbf{W}$
 with derivative
\begin{align*}
\frac{d\widehat{\mathbb{P}}_\lambda}{d\mathbb{P}}\,=\,e^{\langle \mathbf{W}, h\rangle  -\frac{1}{2}\|h\|_\mathcal{H}^2     }\,=\,e^{\lambda\mathbf{W}(D^{b,s})  -\frac{1}{2}\lambda^2     }\,.
\end{align*}
Let $\widehat{\mathbb{E}}_\lambda$ denote the expectation with respect to $\widehat{\mathbb{P}}_\lambda$. The Cauchy-Schwarz  inequality yields that 
\begin{align}
 \mathbb{E}\bigg[\Big(M_{\beta}\big(\mathbf{W},\Gamma^{b,s}\big) \Big)^{\frac{1}{2}}   \bigg]\,=\,&\widehat{\mathbb{E}}_\lambda\bigg[\Big(M_{\beta}\big(\mathbf{W},\Gamma^{b,s}\big)\Big)^{\frac{1}{2}}  e^{-\lambda\mathbf{W}(D^{b,s})  +\frac{1}{2}\lambda^2     }  \bigg]\nonumber \\  \,\leq \,&\widehat{\mathbb{E}}_\lambda\Big[M_{\beta}\big(\mathbf{W},\Gamma^{b,s}\big)   \Big]^{\frac{1}{2}}\widehat{\mathbb{E}}\bigg[\Big( e^{-\lambda\mathbf{W}(D^{b,s})  +\frac{1}{2}\lambda^2     } \Big)^2  \bigg]^{\frac{1}{2}}\,.
\nonumber  
\intertext{Since $\widehat{\mathbb{E}}_\lambda\big[ F(\mathbf{W})  \big]\,=\,\mathbb{E}\big[ F(\mathbf{W}+h)  \big]$ for any integrable function $F$ of the field and $M_{\beta}\big(\mathbf{W}+h,dp\big)= e^{\beta\langle Y_p|h\rangle}  M_{\beta}\big(\mathbf{W},dp\big)$ where $\langle Y_p|h\rangle=\lambda$, the above is equal to    }
  \, = \,&\mathbb{E}\Big[e^{\lambda\beta}M_{\beta}\big(\mathbf{W},\Gamma^{b,s}\big)   \Big]^{\frac{1}{2}} \mathbb{E}\Big[ e^{-\lambda\mathbf{W}(D^{b,s})  +\frac{1}{2}\lambda^2     }  \Big]^{\frac{1}{2}}\nonumber   \\  \, = \,&e^{\frac{1}{2}\lambda\beta+\frac{1}{2}\lambda^2 } \,.
\end{align}
The  above is minimized as $\exp\{-\frac{1}{8}\beta^2\}$ when $\lambda=-\frac{1}{2}\beta$, and thus tends to zero as $\beta$ grows.

\end{proof}

\vspace{.3cm}

The proof below also borrows ideas from~\cite{Lacoin}.

\begin{proof}[Proof of part (ii) of Theorem~\ref{ThmDisorder}]Fix $n\in \mathbb{N}$,  $\mathbf{q},\mathbf{p}\in \Gamma^{b,s}_n$ with $\mathbf{q}\neq\mathbf{p}$, and $\alpha>1$.   It suffices to show that as $n\rightarrow \infty$  
\begin{align}\label{Yeb}
\mathbb{P}\bigg[ \frac{\nu_\beta\big(\mathbf{W}, \mathbf{p} \big)   }{ \nu_\beta\big(\mathbf{W},  \mathbf{q}\big)   }\in  \big( \lambda^{-1}, \lambda   \big]     \bigg]\,\longrightarrow \,0\,.
\end{align}
The analysis below shows that there exist $c,C>0$ such that for all $\beta>1$ 
\begin{align}\label{ShowThis}
\mathbb{P}\bigg[ \frac{\nu_\beta\big(\mathbf{W}, \mathbf{p} \big)   }{ \nu_\beta\big(\mathbf{W},  \mathbf{q}\big)   }\in  \big( \lambda^{-1}, \lambda   \big]     \bigg]\,\leq  \,\min_{\substack{m\in \Z \\  |m| \leq c\log(\beta) }  }  C\sqrt{\mathbb{P}\bigg[ \frac{\nu_\beta\big(\mathbf{W}, \mathbf{p} \big)   }{ \nu_\beta\big(\mathbf{W},  \mathbf{q}\big)   }\in  \big( \lambda^{2m-1}, \lambda^{2m+1}   \big]     \bigg]}\,.
\end{align}
Since the terms $\mathbb{P}\Big[ \frac{\nu_\beta (\mathbf{W}, \mathbf{p} )   }{ \nu_\beta (\mathbf{W},  \mathbf{q} )   }\in  ( \lambda^{2m-1}, \lambda^{2m+1}   ]\Big] $ sum to $1$ over $m\in \mathbb{Z}$, the above must be smaller than $C(1/\lfloor c\log\beta  \rfloor) ^{\frac{1}{2}}$, thus implying~(\ref{Yeb}).

Next I will show~(\ref{ShowThis}).  Define $h\in L^{2}(D^{b,s},\nu)$ as  $h=\alpha\chi\big( \cup_{ \mathbf{p}(k) \neq \mathbf{q}(k) }\mathbf{p}(k)    \big) -  \alpha\chi\big(\cup_{ \mathbf{p}(k) \neq \mathbf{q}(k) }\mathbf{q}(k)    \big)$ for a parameter $\alpha\in [-1,1]$.  Then for any $p\in \mathbf{p}$ and $q\in \mathbf{p}$
\begin{align}
\langle h, Y_p\rangle\,=\,\frac{\alpha\mathbf{n}}{s^n}\hspace{1cm}\text{and}\hspace{1cm}\langle h, Y_q\rangle\,=\,-\frac{\alpha\mathbf{n}}{s^n}\,,
\end{align}
where $1\leq \mathbf{n}<s^n$ is the number of edges not shared by the paths $\mathbf{p},\mathbf{q}\in \Gamma^{b,s}_n$. Notice that 
\begin{align}\label{Note}
 \frac{\nu_\beta\big(\mathbf{W}+h,  \mathbf{p} \big)   }{ \nu_\beta\big(\mathbf{W}+h,  \mathbf{q}\big)   }\,=\, \frac{M_\beta\big(\mathbf{W}+h,  \mathbf{p} \big)   }{ M_\beta\big( \mathbf{W}+h, \mathbf{q}\big)   }\,=\, e^{  \beta\langle h, Y_p\rangle-\beta\langle h, Y_q\rangle }\frac{M_\beta\big(\mathbf{W},  \mathbf{p} \big)   }{ M_\beta\big( \mathbf{W}, \mathbf{q}\big)   }\,=\,e^{2\beta\alpha\mathbf{n}   }\frac{\nu_\beta\big(\mathbf{W},  \mathbf{p} \big)   }{ \nu_\beta\big(\mathbf{W},  \mathbf{q}\big)   }\,.
\end{align}
Define $\widehat{\mathbb{P}}$ to have derivative $\frac{ \widehat{d\mathbb{P}} }{ d\mathbb{P}  }= \textup{exp}\{\langle \mathbf{W}, h\rangle  -\frac{1}{2}\|h\|_\mathcal{H}^2     \}$.  Applying the Cauchy-Schwarz inequaly,
\begin{align}
\mathbb{P}\bigg[\frac{\nu_\beta\big(\mathbf{W},  \mathbf{p} \big)   }{ \nu_\beta\big(\mathbf{W},  \mathbf{q}\big)   }\in  \big[ \lambda^{-1}, \lambda   \big]     \bigg]\,=\,&\widehat{\mathbb{E}}\bigg[e^{-\langle \mathbf{W}, h\rangle  +\frac{1}{2}\|h\|_\mathcal{H}^2     }\chi\bigg(\frac{\nu_\beta\big(\mathbf{W},  \mathbf{p}  \big)   }{ \nu_\beta\big(\mathbf{W}, \mathbf{q} \big)   }\in  \big[ \lambda^{-1}, \lambda   \big] \bigg)    \bigg]  \nonumber \\ \leq \,&\widehat{\mathbb{E}}\Bigg[\Big(e^{-\langle \mathbf{W}, h\rangle  +\frac{1}{2}\|h\|_\mathcal{H}^2     }\Big)^2\bigg]^{\frac{1}{2}} \widehat{\mathbb{P}}\bigg[\frac{\nu_\beta\big(\mathbf{W},  \mathbf{p} \big)   }{ \nu_\beta\big(\mathbf{W}, \mathbf{q}\big)  }\in  \big[ \lambda^{-1}, \lambda   \big]    \bigg]^{\frac{1}{2}}\,.  \nonumber  \intertext{Since the law $\widehat{\mathbb{P}}$ is a shift of $\mathbb{P}$ by $h$, the above is equal to }
 = \,&e^{\frac{1}{2}\|h\|_\mathcal{H}^2     } \mathbb{P}\bigg[\frac{\nu_\beta\big(\mathbf{W}+h,  \mathbf{p} \big)   }{ \nu_\beta\big(\mathbf{W}+h, \mathbf{q}\big)  }\in  \big[ \lambda^{-1}, \lambda   \big]    \bigg]^{\frac{1}{2}}
\nonumber  \\ = \,&e^{ \alpha^2\frac{\mathbf{n}}{(bs)^n}    } \mathbb{P}\bigg[e^{2\beta\alpha\mathbf{n}   }\frac{\nu_\beta\big(\mathbf{W},  \mathbf{p} \big)   }{ \nu_\beta\big(\mathbf{W},  \mathbf{q}\big)   }\in  \big[ \lambda^{-1}, \lambda   \big]    \bigg]^{\frac{1}{2}}\,,\nonumber 
\end{align}
where the second inequality is by~(\ref{Note}).  With $\alpha$ ranging over $[-1,1]$, the above implies~(\ref{ShowThis}) with $C:=\exp\{ 1/b^n \}$ and $c:=\log(2\mathbf{n})/\log (\lambda)      $.

\end{proof}

\subsection{Chaos expansion}

The proof of the proposition below is in the Appendix.  

\begin{proposition}\label{PropUnif} Let $S\subset E^{b,s}$ be finite and  $\Gamma_S^{b,s}$ be nonempty. Define the measure $\mu_{S}^{(n)}$ such that
\begin{align}
\mu_{S}^{(n)}(A)\,=\,\frac{ \mu\big(A\cap G_S^{(n)}  \big)  }{ \mu\big(G_S^{(n)}\big)  } \,,
\end{align}
where $G_S^{(n)}$ is the set of $p\in \Gamma^{b,s}$ such that $S\subset \cup_{k=1}^{s^n}[p]_n(k)$.  Then the sequence $\big\{\mu_{S}^{(n)}\big\}_{n\in \mathbb{N}}$ converges vaguely to a limiting probability measure $\mu_{S}$.

\end{proposition}

The following defines a   generalization of the measure $\rho_k(x_1,\ldots, x_k; dp)$ on $\Gamma^{b,s}$; see Definition~\ref{DefRho}.

\begin{definition} Let $x_1,\ldots,x_k$ be distinct elements in $E^{b,s}$.  If $\Gamma^{b,s}_{\{x_1,\ldots,x_k\}}$ is nonempty,  define the measure $\rho^{(n)}_k(x_1,\ldots, x_k; ds)$ such that for  a Borel set $A\subset \Gamma^{b,s}$ 
\begin{align}
 \rho^{(n)}_k(x_1,\ldots, x_k; A)\,=\,b^{\gamma^{(n)}(\{x_1,\ldots,x_k\})}\mu_{\{x_1,\ldots, x_k\}}^{(n)}(A)\,,
 \end{align}
where $\gamma^{(n)}(S)$ is defined for $n\in \mathbb{N}$  and a finite set $S \subset E^{b,s} $ as follows:
$$ \gamma^{(n)} (S)\,=\,\sum_{k=0}^{n-1} \Big(|S|\,-\, \big| \big\{\mathbf{e}\in E_{k}^{b,s}\,\big|\, \mathbf{e}\cap S\}\neq \emptyset  \big\}\big| \Big) \, .  $$  If $\Gamma^{b,s}_{\{x_1,\ldots,x_k\}}$ is empty, I define $\rho^{(n)}_k(x_1,\ldots, x_k; dp)=0$.

\end{definition}

\begin{corollary}\label{CorUnifConv} For fixed, distinct $x_1,\ldots,x_k\in E^{b,s}$, the sequence of measures $\big\{ \rho^{(n)}_k(x_1,\ldots, x_k; dp)\big\}_{n\in\mathbb{N}}$ converges vaguely to $ \rho_k(x_1,\ldots, x_k; dp)$.  Moreover, if $A\subset \Gamma^{b,s}_N$, then 
$ \rho^{(n)}_k(x_1,\ldots, x_k; A)= \rho_k(x_1,\ldots, x_k; A)$ for all $n>N$.

\end{corollary}

\begin{remark}  I drop superscripts in the following cases: $ \gamma\equiv \gamma^{(\infty)}$, $\rho_k \equiv \rho_k^{(\infty)}$, and identify  $\rho_k(x_1,\ldots, x_k)\equiv \rho_k(x_1,\ldots, x_k; \Gamma^{b,s}) $. 
\end{remark}

\begin{remark}\label{RemarkDen} The measure $\rho^{(n)}_k(x_1,\ldots, x_k; dp)$ is equal to $r^{(n)}_k(x_1,\ldots, x_k; p)\mu(dp)$, where the density
can be written as
 $$r^{(n)}_k(x_1,\ldots, x_k; p)\, = \, b^{\gamma^{(n)} }\frac{\chi_{G_{\{x_1,\ldots,x_k\} }^{(n)}  }(p)}{\mu\big(G_{\{x_1,\ldots,x_k\} }^{(n)} \big)   }\, = \,b^{nk}\chi\bigg( \{x_1,\ldots,x_k\}\subset \bigcup_{\ell=1}^{s^n}[p]_n(\ell)     \bigg)\,=\,\prod_{\ell=1}^k Y_p^{(n)}(x_\ell) \,. $$

\end{remark}

\begin{proposition}\label{PropChaosExp} Let $M^{(n)}_\beta$  be the GMC measure over the field $(\mathbf{W}, \beta Y^{(n)})$.    For a Borel set $A\subset \Gamma^{b,s}$, the random variable $M^{(n)}_\beta(A)$ has the chaos expansion
\begin{align}\label{Hebble}
\,\mu(A)\,+\,\sum_{k=1}^{\infty}\frac{\beta^k}{k!}\int_{(D^{b,s})^k }\rho_k^{(n)}(x_1,\ldots, x_k; A)\mathbf{W}(x_1)\cdots \mathbf{W}(x_k)  \nu(dx_1)\cdots\nu(dx_k) \,.
\end{align}
 The hierarchy of functions $\{\rho_k^{(n)}(x_1,\ldots, x_k; A)\}_{k\in \mathbb{N}}$ satisfies 
\begin{enumerate}[(I).]
\item $ \int_{D^{b,s}}\rho_k^{(n)}\big(x_1,\ldots, x_k; A\big) \nu(dx_k)\,=\,\rho_{k-1}^{(n)}(x_1,\ldots, x_{k-1}; A) $ and
 
\item $\int_{(D^{b,s})^k}\rho_k^{(n)}\big(x_1,\ldots, x_k; A\big)\nu(dx_1)\cdots \nu(dx_k)\,=\,\mu(A)$.
\end{enumerate}

\end{proposition}
\begin{proof}Recall that 
\begin{align}\label{Gradle}
M^{(n)}_\beta(A)\,=\,\int_{A} e^{\beta\langle \mathbf{W}, Y_p^{(n)}\rangle -\frac{\beta^2}{2}\| Y_p^{(n)} \|^2_{\mathcal{H}}    }\mu(dp)\,,
\end{align}
where $Y_p^{(n)}\in \mathcal{H}$ is equal to  $Y_p^{(n)}= b^n\chi\big( \bigcup_{k=1}^{s^n}[p]_n(k)  \big)  $.  The integrand above has the chaos expansion
\begin{align*}
e^{\beta\langle \mathbf{W}, Y_p^{(n)}\rangle -\frac{\beta^2}{2}\| Y_p^{(n)} \|^2_{\mathcal{H}}    }\,=\,1\,+\,\sum_{k=1}^{\infty}\frac{\beta^k}{k!}\int_{(D^{b,s})^k }\prod_{\ell=1}^k  Y_p^{(n)}(x_\ell) \mathbf{W}(x_1)\cdots \mathbf{W}(x_k)  \nu(dx_1)\cdots\nu(dx_k)\,.
\end{align*}
By Remark~\ref{RemarkDen}, $\prod_{\ell=1}^k  Y_p^{(n)}(x_\ell)$ is equal to the density, $r_k^{(n)}(x_1,\ldots, x_k; p)$, of the measure $\rho_k^{(n)}(x_1,\ldots, x_k; dp)$.  Thus, the expressions~(\ref{Hebble}) and~(\ref{Gradle})  are equal.\vspace{.4cm}

The equalities~(I) and (II) also  follow from Remark~\ref{RemarkDen} since, for instance,
\begin{align*}
\int_{D^{b,s}}\rho_k^{(n)}(x_1,\ldots, x_k; A)\nu(dx_k)\,=\,&\int_{D^{b,s}}\int_{A}r_k^{(n)}(x_1,\ldots, x_k; p)    \mu(dp)\nu(dx_k)\\ \,=\,&\int_{D^{b,s}} \int_{A} \prod_{\ell=1}^k Y_p^{(n)}(x_\ell)       \mu(dp) \nu(dx_k)\,,\intertext{and switching the order of integration and using that $\int_{D^{b,s}  } Y_p^{(n)}(x)\nu(dx)=1$ yields    } \,=\,&\int_{A} \prod_{\ell=1}^{k-1} Y_p^{(n)}(x_\ell)       \mu(dp) \,=\,\rho_k^{(n)}(x_1,\ldots, x_{k-1}; A) \,.
\end{align*}

\end{proof}

\begin{proof}[Proof of Theorem~\ref{ThmChaosExp}]  By Proposition~\ref{PropYField}, the sequence $\{M^{(n)}_\beta(A)\}_{n\in \mathbb{N}}$ converges in $L^2(\Gamma^{b,s},\mathcal{B}_\Gamma, \mu)$ to $M_\beta(A)$.  Thus,  $$\big\langle  \rho_k^{(n)}(\cdot, A)  \big\rangle\equiv\big(\mu(A),\, \rho_1^{(n)}(x_1;A),\, \rho_2^{(n)}(x_1,x_2; A),\ldots\big) ,$$ viewed as an element of 
$$ L^2\bigg( \bigcup_{k=0}^\infty (D^{b,s})^k ,\, \bigoplus_{k=0}^\infty\mathcal{B}_{D}^{\otimes^k}   ,\,\bigoplus_{k=0}^{\infty} \frac{\beta^{2k}}{k!}\nu^{k}    \bigg)     \,, $$
converges as $n\rightarrow \infty$ to a limit $\langle  \widetilde{\rho}_k (\cdot, A) \rangle$.  However, if $ A\in \mathcal{B}_{\Gamma}^{(n)}:=\mathcal{P}(\Gamma^{b,s}_n)  $, Corollary~\ref{CorUnifConv} implies that  $\langle  \widetilde{\rho}_k (\cdot, A) \rangle=\langle  \rho_k (\cdot, A) \rangle$

\end{proof}

\begin{appendix}

\section{Further discussion of the  diamond hierarchical lattice}\label{AppendCG}

The discussion in this section is applicable to any $b,s\in \{2,3,\ldots\}$.

\subsection{The vertex set}
I will show how  the set of vertices, $V_{n}^{b,s}$, on the graph $ D_{n}^{b,s}$ are embedded in $ D^{b,s}$.  For $n\in \mathbb{N}$, I can label elements in   $V_{n}^{b,s}\backslash V_{n-1}^{b,s}$  by
\begin{align*}
 V_{n}^{b,s}\backslash V_{n-1}^{b,s}\, \equiv\, &    \underbrace{\big( \{1,\ldots, \bbf\}\times \{1,\ldots, \sbf\} \big)^{n-1}}\times \big( \{1,\ldots, \bbf\}\times \{1,\ldots, \sbf-1\}\big) \,.\\ &\hspace{1.85cm} \equiv E^{b,s}_{n-1}
 \end{align*}

  Given an element $v=(b_{1},s_{1})\times \cdots \times (b_{n},s_{n}) \in  V_{n}^{b,s}\backslash V_{n-1}^{b,s} $, define   $ U_{v}=L_v\cup R_{v     }  \subset \mathcal{D}^{\bbf,\sbf}$ for
$$
 L_{v     } \, :=  \bigg\{  (b_{1},s_{1})\times \cdots \times (b_{n},s_{n}) \times \prod_{j=1}^{\infty} (\widehat{b}_{j},\sbf)\,\bigg|\,\widehat{b}_{j}\in \{1,\ldots, b\} \bigg\}\,\subset \, \mathcal{D}^{\bbf,\sbf}   $$
and
$$ R_{v     } \, := \,\bigg\{  (b_{1},s_{1})\times \cdots \times (b_{n},s_{n}+1)\times \prod_{j=1}^{\infty} (\widehat{b}_{j},1) \,\bigg|\,\widehat{b}_{j}\in \{1,\ldots, b\}  \bigg\}\,\subset \, \mathcal{D}^{\bbf,\sbf} \,.
$$
Pairs $x,y\in U_{v}$ satisfy $d_D(x,y )=0$, and $U_{v}$ is the maximal equivalence class with that property.  Thus, $v  $ is  canonically identified with an element in $D^{b,s}$.   The root vertices $V^{b,s}_0=\{A,B\}$ of the graph are identified with the subsets of $\mathcal{D}^{\bbf,\sbf}$ given by 
$$
A\,:= \,   \bigg\{  \prod_{j=1}^{\infty} (\widehat{b}_{j},1) \,\Big|\,  \widehat{b}_{j}\in  \{1,\ldots, b\} \bigg\}\hspace{.5cm}\text{and}\hspace{.5cm}B\,:= \,   \bigg\{  \prod_{j=1}^{\infty} (\widehat{b}_{j},\sbf ) \,\Big|\,  \widehat{b}_{j}\in \{1,\ldots, b\}  \bigg\}\,.
$$

\subsection{The metric space $D^{b,s}$  }

Next I prove the points in Proposition~\ref{PropCompact}.  Note that each element of $E^{b,s}$ is equivalent to a nested sequence $e_n\in E^{b,s}_n$.   \vspace{.2cm}

\noindent \textbf{Completeness:} Let $\{x_k\}_{k\in \mathbb{N}}$ be a Cauchy sequence in $D^{b,s}$.   The sequence  $\widetilde{\pi}(x_k)\in [0,1]$ must be Cauchy and thus convergent to a limit $\lambda\in [0,1]$. \vspace{.2cm}

 If $\lambda$ is a multiple of $ b^{-N}$ for some nonnegative integer $N\in \{0,1,2,\ldots\}$, let $N$ be the smallest such value.  For large $k$, the terms $x_k$ must become arbitrarily close to generation $N$ vertices.  Since the generation $N$ vertices are at least a distance $ 1/s^N$ apart, the terms must be close to the same generation $N$ vertex and thus convergent. \vspace{.2cm}

  If $\lambda$ is not a multiple of $ b^{-N}$ for any $N\in \mathbb{N}$, then there must exist a nested sequence  of edge sets $\mathbf{e}_n\in \cup_{k=1}^\infty E_k^{b,s}$ such that each closure $\overline{\mathbf{e}}_n$ contains a tail of the sequence $\{x_k\}_{k\in \mathbb{N}}$.  The sequence $\{x_k\}_{k\in \mathbb{N}}$ converges to the unique element in the set $\cap_{n=1}^{\infty}\overline{\mathbf{e}}_n$.

\vspace{.3cm}

\noindent \textbf{Compactness:}  Let $\{x_k\}_{k\in \mathbb{N}}$ be a sequence in $D^{b,s}$. Since $D^{b,s}$ is covered by sets $\overline{\mathbf{e}}$ for  $\mathbf{e}\in E^{b,s}_n$, the pigeonhole principle implies that there must exist a  nested sequence of edge sets $\mathbf{e}_n\in E^{b,s}_n$ such that each $\overline{\mathbf{e}}_n$ contains $x_k$ for infinitely many $k\in \mathbb{N}$.  Thus, there is a subsequence of $\{x_k\}_{k\in \mathbb{N}}$ converging to the unique element in $\cap_{n=1}^{\infty}\overline{\mathbf{e}}_n$.

\vspace{.3cm}

\noindent \textbf{Hausdorff dimension:} The contractive maps $S_{i,j}$'s defined in Remark~\ref{RemarkFractal} are the similitudes of the fractal $D^{b,s} $.    The open set $O:=D^{b,s}\backslash \{A,B\}$ is a separating set since
\begin{align*}
\bigcup_{i,j}S_{i,j}(O) \, \subset \, O  \hspace{1cm}\text{and} \hspace{1cm} S_{i,j}(O)\cap S_{k,l}(O)\,=\,\emptyset \, 
\end{align*}
 for  $ (i,j)\neq (k,l)$.      Since the $S_{i,j}$'s have contraction constant $1/s$ and there are $b s$ maps, the Hausdorff dimension of $D^{b,s}$ is $\log(bs)/\log s$;  see~\cite[Section 11.2]{Folland} for  a discussion of Hausdorff dimension and self-similarity.

\subsection{The measures}

Finally, I prove Propositions~\ref{PropCylinder},~\ref{PropUniform}, and~\ref{PropUnif}. \vspace{.3cm}

\noindent \textbf{Uniform measure on the diamond lattice:}  Let $\mathcal{B}_E :=\{A
\cap E^{b,s}=\,|\, A\in \mathcal{B}_D \}     $ be the restriction of $\mathcal{B}_D$ to $E^{b,s}=D^{b,s}\backslash V^{b,s}$.  Every element $A\in  \mathcal{B}_D$ can be decomposed as $A=A_1\cup A_2$ for $A_1 \subset V^{b,s}$ and $A_2\in \mathcal{B}_E$. I define $\nu$ as zero on  $V^{b,s}$, and it remains to define $\nu$ on $\mathcal{B}_E$.  The Borel $\sigma$-algebra  $\mathcal{B}_E$ is generated by the semi-algebra, $\mathcal{C}_E=\cup_{n=0}^{\infty} E_n^{b,s}  $, i.e., the collection of cylinder sets $C_{(b_1,s_1)\times \cdots \times (b_n,s_n)}$.  This follows since for any open set $O\subset D^{b,s}$ we can write 
\begin{align}\label{OpenSet}
O\backslash V^{b,s}\,=\,\bigcup_{x\in O\backslash V^{b,s}} \mathbf{e}_x  \,, 
\end{align}
where $\mathbf{e}_x\in \mathcal{C}_E$ is the biggest set in $\mathcal{C}_E$ satisfying  $x\in \mathbf{e}_x\subset O$.  A  premeasure $\widehat{\nu}$ can be placed on $ \mathcal{C}_E $ by assigning $\widehat{\nu}(\mathbf{e})=| E_n^{b,s}|^{-1}= (bs)^{-n}$ to each   $\mathbf{e}\in E_n^{b,s}  $.  The finite premeasure $ (E^{b,s}, \mathcal{C}_E,\widehat{\nu})$ extends uniquely to a measure $(E^{b,s},\mathcal{B}_E, \nu)$ through the Carath\'eodory procedure. \vspace{.5cm}

\noindent \textbf{Uniform measure on paths:} Consider the semi-algebra $\mathcal{C}_\Gamma:=\cup_{n=0}^{\infty}\Gamma^{b,s}_n$ of subsets of $\Gamma^{b,s}$.  An arbitrary open set $O\subset \Gamma^{b,s}$ can be written as a  disjoint union of elements in $\mathcal{C}_\Gamma$ in analogy to~(\ref{OpenSet}).  Indeed, each element in $\mathcal{C}_\Gamma$ is an open ball with respect to the metric $d_\Gamma$.   A finite premeasure $\widehat{\mu}$ is defined on $ \mathcal{C}_\Gamma $ by assigning each $q\in \Gamma^{b,s}_n$ the value $\widehat{\mu}(q)=|\Gamma^{b,s}_n|^{-1}$.  Again, by  Carath\'eodory's technique, the measure $\widehat{\mu}$ extends to a measure $(  \Gamma^{b,s},\mathcal{B}_\Gamma, \mu)$. 

Next I will argue that $\mu$ and $\nu$ are related through the identity~(\ref{MuToNu}). To construct a  measure $\widetilde{\nu}$ on $D^{b,s}$ satisfying $\widetilde{\nu}(R)=\int_{\mu}\int_0^1 1_{R}\big(p(r)\big)\mu(dp)  $, I can follow the approach  used above for constructing $\nu$  by defining how  $\widetilde{\nu}$ acts on $V^{b,s}$ and the cylinder sets. In particular, if $\widetilde{\nu}$ agrees with $\nu$ on $V^{b,s}$ and the cylinder sets, then they must be equal.   The vertex set $V^{b,s}$ has $\widetilde{\nu}$-measure zero because a path $p:[0,1]\rightarrow D^{b,s}$ only passes  through vertex points at the countable collection of times $r\in [0,1]$ that are integer multiples of $s^{-n}$ for some $n\in \mathbb{N}$, and thus $\int_0^1 1_{V^{b,s}}\big(p(r)\big)=0$ for all $p\in \Gamma^{b,s}$.  A cylinder set  $R=C_{(b_1,s_1)\times \cdots \times (b_n,s_n)}$ has   $\widetilde{\nu}$-measure $(bs)^{-n}$ since a path chosen uniformly at random will pass through $R$ with probability $b^{-n}$ and, in that event, it will spend a duration of $s^{-n}$ in $R$.   Therefore $\widetilde{\nu}=\nu$.

\vspace{.5cm}

\noindent \textbf{Uniform measure on paths through a finite subset of $\mathbf{E^{b,s}}$:} Let $S\subset E^{b,s}$ be  finite and $\Gamma^{b,s}_S$ be nonempty.  There exists an $N\in \mathbb{N}$ such that no two elements in $S$ fall into the same equivalence class $\mathbf{e}\in E^{b,s}_N$.  For $n>N$, 
$$ \frac{d\mu^{(n)}_S  }{ d\mu  }\,=\,J_S^{(n)}\hspace{.5cm}\text{where}\hspace{.5cm}   J_S^{(n)}(p)\,:=\,  b^{n|S| -\gamma(S) }\chi\Big([p]_n\cap \Gamma^{b,s}_S \neq \emptyset   \Big)   \,.  $$
Moreover, $J_S^{(n)}$ forms a nonnegative, mean-one martingale with respect to the filtration $\textup{F}_n=\Gamma^{b,s}_n   $.  If $g:\Gamma^{b,s}\rightarrow \R$ is measurable with respect to $\textup{F}_m$ for some $m\in \mathbb{N}$, then the sequence $\int_{\Gamma^{b,s}  } g(p) \mu^{(n)}(dp) $ is constant for $n\geq m$ and thus convergent.    A continuous function  $h:\Gamma^{b,s}\rightarrow \R$ must be uniformly continuous  since $\Gamma^{b,s}$ is a compact metric space.  Thus, given $\epsilon>0$, there exists an $n\in \mathbb{N}$ such that $|g(p)-g(q)|<\epsilon $ when $[p]_n=[q]_n$.  It follows that $Y_n g= \int_{\Gamma^{b,s}  } g(p) \mu^{(n)}(dp)    $ converges to a limit $Y_\infty g$ as $n\rightarrow \infty$.  Therefore, $\mu^{(n)}_S$ converges vaguely to a limiting probability measure $\mu_S$ on $\Gamma^{b,s}$.

\section{Random shift}

The following argument from~\cite{Shamov} shows that $\beta Y$ defines a random shift. 

\begin{proof}[Proof of Theorem~\ref{ThmRandomShift}] Let $\widetilde{\mathbb{E}}_\beta$ refer to the expectation with respect to $\widetilde{\mathbb{P}}_{\beta}$. The calculation below shows that   $\widetilde{\mathbb{E}}_{\beta}\big[  e^{\langle \mathbf{W},\psi \rangle} \big]=\mathbb{E}_{\beta}\big[ M_\beta(\mathbf{W},\Gamma^{b,s}) e^{\langle \mathbf{W},\psi \rangle} \big]$ for any $\psi \in \mathcal{H}$, and thus that  $ M_\beta(\mathbf{W},\Gamma^{b,s})$  is the Radon-Nikodym derivative of $\widetilde{\mathbb{P}}_{\beta}$ with respect to $\mathbb{P}$.
\begin{align*}
\widetilde{\mathbb{E}}_{\beta}\big[  e^{\langle \mathbf{W},\psi \rangle} \big]\,:=&\, \int_{\Gamma^{b,s}}\mathbb{E}\big[  e^{\langle \mathbf{W}+\beta Y_p,\psi \rangle} \big]\mu(dp) \\ \,=&\, \mathbb{E}\big[  e^{\langle \mathbf{W},\psi \rangle} \big]\int_{\Gamma^{b,s}} e^{\beta \langle  Y_p,\psi \rangle} \mu(dp) \,
\intertext{Since $Y^{(n)}$ converges strongly to $Y$ by part (i) of Proposition~\ref{PropYs}, the above is equal to }
 \,=&\,   e^{\frac{1}{2} \| \psi\|_{\mathcal{H}}^2} \lim_{n\rightarrow \infty} \int_{\Gamma^{b,s}} e^{\beta \langle  Y_p^{(n)},\psi \rangle} \mu(dp)
\\ \,=&\,    \lim_{n\rightarrow \infty} \int_{\Gamma^{b,s}} e^{ \frac{1}{2}\| \psi+ \beta Y_p^{(n)}\|^2_\mathcal{H} -\frac{\beta^2}{2}\| Y_p^{(n)}\|^2_\mathcal{H} } \mu(dp)\,.
\intertext{With $\mathbb{E}\big[ \exp\{ \langle \mathbf{W} ,\phi\rangle  \} \big]=\exp\big\{\frac{1}{2}\|\phi\|_2^2   \big\}    $, I have the equality    }  
    \,=&\, \lim_{n\rightarrow \infty}\int_{\Gamma^{b,s}}\mathbb{E}\big[  e^{\langle \mathbf{W},\psi +\beta Y_p^{(n)}\rangle} \big]e^{ -\frac{\beta^2}{2}\| Y_p^{(n)}\|^2_\mathcal{H} } \mu(dp)  \\
   \,=&\, \lim_{n\rightarrow \infty}\mathbb{E}\bigg[ \bigg(\int_{\Gamma^{b,s}}e^{ \beta\langle \mathbf{W} , Y_p^{(n)} \rangle -\frac{\beta^2}{2}\| Y_p^{(n)}   \|^2_\mathcal{H}  } \mu(dp) \bigg) e^{\langle \mathbf{W},\psi \rangle} \bigg]\,=\, \lim_{n\rightarrow \infty}\mathbb{E}\Big[ M^{(n)}_\beta\big(\mathbf{W},\Gamma^{b,s}\big)  e^{\langle \mathbf{W},\psi \rangle} \Big]\,\,.
\intertext{Finally, the limit can be brought inside the expectation as a consequence of part (ii) of Proposition~\ref{PropYField}   } 
   \,=&\,\mathbb{E}\Big[ M_\beta\big(\mathbf{W},\Gamma^{b,s}\big)  e^{\langle \mathbf{W},\psi \rangle} \Big]\,.
\end{align*}

\end{proof}

\end{appendix}

\end{document}